\documentclass[a4paper,11pt,reqno]{amsart}
\usepackage[left=1in,right=1in,top=1in,bottom=1in]{geometry}
\usepackage{enumitem}
\usepackage{amssymb, amsmath,latexsym,amsfonts,amsbsy, amsthm,mathtools,graphicx,color}
\usepackage{float}
\usepackage{tikz}
\usetikzlibrary{arrows, decorations.markings,matrix,trees}
\usepackage{hyperref}
\hypersetup{
   colorlinks=true,
   citecolor=blue,
   filecolor=blue,
   linkcolor=blue,
   urlcolor=blue
}
\def\Xint#1{\mathchoice
  {\XXint\displaystyle\textstyle{#1}}%
  {\XXint\textstyle\scriptstyle{#1}}%
  {\XXint\scriptstyle\scriptscriptstyle{#1}}%
  {\XXint\scriptscriptstyle\scriptscriptstyle{#1}}%
  \!\int}
\def\XXint#1#2#3{{\setbox0=\hbox{$#1{#2#3}{\int}$}
  \vcenter{\hbox{$#2#3$}}\kern-.5\wd0}}

\def\dashint{\Xint-}
\newcommand{\al}{\alpha}       
    
\newcommand{\lda}{\lambda}
\newcommand{\om}{\Omega}            
\newcommand{\pa}{\partial}
\newcommand{\va}{\varepsilon}       
\newcommand{\ud}{\mathrm{d}}
\newcommand{\be}{\begin{equation}} 
\newcommand{\ee}{\end{equation}}
\newcommand{\w}{\omega}      
\newcommand{\Lda}{\Lambda}

\newcommand{\bA}{\mathbb{A}}

\newcommand{\cB}{\mathcal{B}}

\newcommand{\cC}{\mathcal{C}}

\newcommand{\rD}{\mathrm{D}}
\newcommand{\cD}{\mathcal{D}}

\newcommand{\cF}{\mathcal{F}}

\newcommand{\bG}{\mathbb{G}}

\newcommand{\cG}{\mathcal{G}}

\newcommand{\cL}{\mathcal{L}} 

\newcommand{\Z}{\mathbb{Z}}

\newcommand{\M}{\mathcal{M}}
\newcommand{\MM}{\mathbb{M}}

\newcommand{\cN}{\mathcal{N}}

\newcommand{\cP}{\mathcal{P}}

\newcommand{\R}{\mathbb{R}}

\newcommand{\cS}{\mathcal{S}} 
\newcommand{\Ss}{\mathbb{S}}

\newcommand{\cT}{\mathcal{T}}

\newcommand{\cW}{\mathcal{W}}

\newcommand{\wc}{\rightharpoonup}        

\newcommand{\HH}{\mathcal{H}}

\newcommand{\vp}{\varphi}

\newcommand{\ga}{\gamma}

\newcommand{\sg}{\sigma} 
\newcommand{\ift}{\infty} 
\newcommand{\wt}{\widetilde}

\newcommand{\f}{\frac}

\newcommand{\ol}{\overline}

\newcommand{\op}{\operatorname}
\newcommand{\Sg}{\Sigma}

\newcommand{\na}{\nabla}

\DeclareMathOperator{\dist}{dist}

\DeclareMathOperator{\supp}{supp}

\DeclareMathOperator{\sing}{sing}
\DeclareMathOperator{\reg}{reg}

\DeclareMathOperator{\loc}{loc}

\def\<{\langle}\def\>{\rangle}
\def\({\left(}\def\){\right)}
\numberwithin{equation}{section}
\theoremstyle{plain}
\newtheorem{thm}{Theorem}[section]
\newtheorem{cor}[thm]{Corollary}

\newtheorem{lem}[thm]{Lemma}
\newtheorem{prop}[thm]{Proposition}

\theoremstyle{definition}
\newtheorem{defn}[thm]{Definition}

\theoremstyle{remark}
\newtheorem{rem}[thm]{Remark}

\title[Stratification and Rectifiability of Harmonic Map Flows]{Stratification and Rectifiability of Harmonic Map Flows via Tangent Measures}

\author{Haotong Fu}
\address{School of Mathematical Sciences, Peking University, Beijing 100871, China}
\email{2301110012@pku.edu.cn}

\author{Wei Wang}
\address{School of Mathematical Sciences, Peking University, Beijing 100871, China}
\email{gjmtamag@gmail.com,\,\,2201110024@stu.pku.edu.cn}

\author{Ke Wu}
\address{School of Mathematics, Yunnan Normal University, Kunming, 650500, China}
\email{m18629096093@163.com,\,\,kewu@ynnu.edu.cn}

\author{Zhifei Zhang}
\address{School of Mathematical Sciences, Peking University, Beijing 100871, China}
\email{zfzhang@math.pku.edu.cn}

\begin{document}
\begin{abstract}
In this paper, we investigate the stratification theory for ``suitable solutions" of harmonic map flows based on the spatial symmetry of tangent measures. Building on the quantitative stratifications and Reifenberg-rectifiable theory developed by Naber and Valtorta in breakthrough research of harmonic maps \cite{NV17}, we prove that each time slice of the singular set in our model is rectifiable. By making some additional assumptions about the target manifolds to exclude specific tangent flows and measures, we can also obtain a sharp regularity of suitable solutions for harmonic map flows.
\end{abstract}
\maketitle

%\tableofcontents
\section{Introduction}

\subsection{Background}
Let $ \om\subset\R^n $ $ (n\in\Z_{\geq 2}) $ be a bounded smooth domain and  $ \cN $ be a smooth compact manifold isometrically embedded in $ \R^d $ with $ d\in\Z_{\geq 2} $. The harmonic map flow from $ \om$ to $ \cN $ is
\be
\pa_tu-\Delta u=A(u)(\na u,\na u),\label{heatflow}
\ee
where $ A(y)(\cdot,\cdot):T\cN\times T\cN\to (T\cN)^{\perp} $ is the second fundamental form of $ \cN $ at the point $ y\in\cN $. For the Cauchy problem associated with \eqref{heatflow}, we assume
\be
u(\cdot, 0)=u_{0}\in C^{\infty}(\Omega, \cN).\label{initial data}
\ee

Eells and Sampson introduced the concept of harmonic map flow in \cite{ES64} to find harmonic maps in a given free homotopy class. When $ \om $ is a compact, smooth $ n $-manifold without boundary, Eells and Sampson showed that there exist smooth short-time solutions starting from any smooth initial data. Using the Bochner formula, they further revealed that if $ \cN $ has non-positive sectional curvature $ K_{\cN} $, then the solution remains globally smooth and unique. 

On the other hand, without the curvature assumption on $ \cN $, the short-time smooth solution may blow up in finite time. Here we say that a solution of \eqref{heatflow} blows up in finite time if there exists $0<T<+\infty$ such that $u(\cdot, t)$ is smooth in $\Omega\times (0, T)$, while
\[
\lim_{t\uparrow T}\|\na u(\cdot,t)\|_{L^{\ift}(\om)}=+\ift.
\]
In \cite{CDY92}, Chang, Ding, and Ye found a surprising fact: even if the base and target manifolds are both $ \Ss^2 $, solutions of \eqref{heatflow} evolving from smooth initial data can develop singularities in finite time. After that, the construction of blow-up solutions became an active area of research. In a recent breakthrough \cite{DPW20}, D\'{a}vila, del Pino, and Wei constructed an explicit blow-up solution with the desired blow-up profile for the harmonic map flow from  $ \om\subset\R^2 $ to $ \Ss^2 $ with specific initial conditions.

In addition to the classical solutions evolving from smooth initial data, the weak solution is also an essential class of solutions. Unlike classical solutions, the uniqueness of weak solutions cannot be established without additional assumptions. Indeed, Coron \cite{Cor90} gave an example in which the given harmonic map flow has infinitely many weak solutions. If a weak solution has higher a priori regularity, Lin and Wang \cite{LW10} proved that it becomes unique. Freire \cite{Fre95a, Fre95b} later proved the uniqueness of weak solutions with monotonically decreasing energy.

In  \cite{Fel94}, Feldman introduced the notion of suitable weak solutions for \eqref{heatflow}, i.e., weak solutions satisfying a parabolic stationary condition. Motivated by this, we characterize suitable solutions by the following properties.

\begin{defn}[Suitable solution]\label{SuitableSolution}
We call a map $ u:\Omega\times(0,T)\to\cN $ a suitable solution of \eqref{heatflow} if $ \pa_tu,\na u\in L^2(\Omega\times(0,T)) $, and $ u $ satisfies the following properties.
\begin{enumerate}
\item $ u $ is a weak solution of \eqref{heatflow}. That is, for any $ \psi\in C_0^{\ift}(\Omega\times(0,T)) $,
\[
\int_{\Omega\times(0,T)}(\pa_tu\psi+\na u\cdot\na\psi-A(u)(\na u,\na u)\psi)\ud x\ud t=0.
\]
\item $ u $ satisfies the local energy identity: for any $ \theta\in C_0^{\ift}(\Omega\times(0,T)) $, 
\be
\int_{\Omega\times(0,T)}(|\na u|^2\pa_t\theta-2|\pa_tu|^2\theta-2\pa_tu\na u\cdot\na\theta)\ud x\ud t=0.\label{LocalizedEnergy}
\ee
\item $ u $ fulfills the stationary condition: for any $ \xi\in C_0^{\ift}(\Omega\times(0,T),\R^n) $,
\be
\int_{\Omega\times(0,T)}\(|\na u|^2\op{div}_x\xi-2\sum_{i,j=1}^n\pa_i\xi^j\pa_iu\pa_ju-2\pa_tu\na u\cdot \xi\)\ud x \ud t=0.\label{StationaryCondition2}
\ee
\end{enumerate}
Let $\Lambda$ be a positive constant, the set $H_{\Lda}(\Omega\times(0,T),\cN)$ consists of all suitable solutions of \eqref{heatflow} in $\Omega\times(0,T)$ such that $\int_{\Omega\times(0,T)}|\na u|^2\ud x\ud t\leq\Lda.$
\end{defn}

Properties \eqref{LocalizedEnergy} and \eqref{StationaryCondition2} are similar to the stationary condition for harmonic maps (see \cite{NV17}). The use of the stationary condition for harmonic maps
was based on the fact that this condition implies a monotonicity formula. In the case of \eqref{heatflow}, a monotonicity formula also holds if the solutions satisfy \eqref{LocalizedEnergy} and \eqref{StationaryCondition2} (see Proposition \ref{Monotonicity} below).
Moreover, the conditions for defining suitable solutions are natural. Indeed, by using integration by parts, we can check that smooth solutions satisfy \eqref{LocalizedEnergy} and \eqref{StationaryCondition2}. However, a suitable solution may not be smooth everywhere. For a suitable solution $ u:\Omega\times(0,T)\to\cN $ of \eqref{heatflow},  we define the singular set of $ u $ as
\be
\sing(u):=\{X:=(x, t)\in \Omega\times(0,T):~u~\text{is not smooth in any neighborhood of}~X\},\label{SingularSet}
\ee
while the regular set of $ u $ is $ \reg(u):=(\Omega\times(0,T))\backslash\sing(u) $. By definition, it is easy to see that $\sing(u)$ is relatively closed in $\Omega\times(0,T)$.

In the case where the target manifold is $ \cN=\Ss^{d-1}\subset\R^d $, Feldman \cite{Fel94} and Chen-Li-Lin \cite{CLL95} proved the partial regularity of the suitable solutions. That is, there exists a universal positive constant $\va_{0}$ such that
\be
r^{-n}\int_{P_{2r}(X)}|\na u|^2\ud x\ud t <\va_{0}
\ee
implies $u\in C^{\ift}(P_r(X),\cN)$.
Here for $ r>0 $ and $ X=(x,t)\in\R^n\times\R $, we have used the notation that $ P_r(X):=B_r(x)\times(t-r^2,t+r^2) $ denotes the parabolic ball centered at $ X $ with $ B_r(x)$ being the ball in $ \R^n=\R^n\times\{0\}\subset\R^n\times\R $. If $X=0^{n,1}=(0^n,0) $ with $ 0^n $ being the original point of $ \R^n $, we write $P_r:=P_{r}(0^{n,1})$ for simplicity. For maps into homogeneous manifolds, Chen and Wang \cite{CW96} established analogous properties. Later, Liu \cite{Liu03} generalized these results (see Proposition \ref{PartialRegularity} for more information). 

As a consequence of the partial regularity results listed above, Feldman proved that
\[
\dim_{\cP}(\sing(u))\leq n,
\]
and
\[
\dim_{\HH}(\sing(u)\times\{t\})\leq n-2,\text{ for a.e. }t,
\]
where $ \dim_{\HH} $ and $ \dim_{\cP} $ are the Hausdorff and parabolic Hausdorff   dimensions, respectively.

Given the above estimates concerning the size of the singular set, a natural question arises about further structures of $ \sing(u) $. To attack this question, we first introduce the blow-up around a singular point. 
 
 For $ X_0=(x_0,t_0)\in\R^n\times\R $, $ X=(x,t)\in\R^n\times\R $ and $ \lda>0 $, define $ \cP_{X_0,\lda}(X)=(\f{x-x_0}{\lda},\f{t-t_0}{\lda^2}) $. In particular, $ \cP_{X_0,\lda}^{-1}(X)=(x_0+\lda x,t_0+\lda^2t) $.  For a domain $ U\subset\R^n\times\R $, let $ \M(U) $ be the set of all Radon measures on $ U $. Then $ \M(U) $ forms a topological space with the weak${}^*$ topology. 
For  $ \M(P_r(X)) $, we can define a metric consistent with the weak${}^*$ topology. Start by choosing a dense countable subset $ \{f_i\} $ of $ C_0(P_1) $ under the supremum norm $ \|\cdot\|_\ift $. Since $ \{f_i\} $ is dense in $ C_0(P_1) $, the set $ \{f_i\circ\mathcal{P}_{X,r}\} $ is dense in $ C_0(P_r(X)) $. Let
\[
\MM_{X,r}(\mu,\nu):=\sum_{i\in\Z_+}\f{1}{2^i}\cdot\f{r^{-n}\left|\int_{P_r(X)}f_i\circ\cP_{X,r}\ud\mu-\int_{P_r(X)}f_i\circ\cP_{X,r}\ud\nu\right|}{1+r^{-n}\left|\int_{P_r(X)}f_i\circ\cP_{X,r}\ud\mu-\int_{P_r(X)}f_i\circ\cP_{X,r}\ud\nu\right|}
\]
for any $ \mu,\nu\in\M(P_r(X)) $.
For a sequence $ \{\mu_i\}\subset\M(P_r(X)) $, we have
\[
\lim_{i\to+\ift}\MM_{X,r}(\mu_i,\mu)=0\quad\Longleftrightarrow\quad\mu_i\wc^*\mu\text{ in }\M(P_r(X)).
\]

\begin{defn}[Blow-ups]
Given $ r>0 $, an open set $ U\subset\R^n\times\R $ and a point $ X_{0}\in U $, the blow-up of a measurable map $ u:U\to\cN $ at $X_{0} $ with scale $ r $ is $
T_{X_{0},r}u:=u\circ\cP_{X_{0},r}^{-1} $. Similarly, for $ \mu\in\M(U) $, define $
T_{X_{0},r}\mu:=r^{-n}(\mu\circ\cP_{X_{0},r}^{-1}) $.
\end{defn}

Let $ u:P_r(X_0)\to\cN $ be a suitable solution of \eqref{heatflow}, $ X\in P_r(X_0) $ and $\{r_{i}\}$ be a sequence that tends to $0$ as $i\to\infty$. Then 
there exist $w$ and $\mu$ such that up to a subsequence,
\be
\begin{gathered}
w_i:=T_{X,r_i}u\wc w\text{ weakly in }H_{\loc}^1(\R^n\times\R,\cN),\\
\f{1}{2}|\na w_i|^2\ud x\ud t\wc^*\mu:=\f{1}{2}|\na w|^2\ud x\ud t+\nu\text{ in }\M(\R^n\times\R).
\end{gathered}\label{wmunu}
\ee
We call $ w $ and $ \mu $ a tangent flow and a tangent measure of $ u $ at $ X $, respectively. 
 
\subsection{Previous results of stratification theory}

In this subsection, we briefly recall some previous  results for harmonic map flows. First, we introduce several definitions.

\begin{defn}[Backward self-similar]
A measurable map $ w:\R^n\times\R\to\cN $ is called backward self-similar at $ X\in\R^n\times\R $ if, for any $ \lda>0 $, 
\[ T_{X,\lda}w|_{\R^n\times\R_-}=T_{X,1}w|_{\R^n\times\R_-}.
\]
$ \mu\in\M(\R^n\times\R) $ is backward self-similar at $ X $ if, for any $ \lda>0 $, 
\[
(T_{X,\lda}\mu)\llcorner(\R^n\times\R_-)=(T_{X,1}\mu)\llcorner(\R^n\times\R_-),
\]
where $ \R_-=(-\ift,0) $. If $X=0^{n,1}$, we refer to the map or measure as self-similar for simplicity.
\end{defn}
Let $k\in\mathbb{Z}\cap[1, n]$ and $V$ be a $k-$dimensional subspace of  $\mathbb{R}^{n}$. We say $w$ and $ \mu $ are backward invariant with respect to the subspace $ V$ if for each $v\in V$, 
\[
w(y+v,s)=w(y,s)
\]
and
\[
\mu(\cdot+(v, 0))\llcorner(\R^n\times\R_-)=\mu(\cdot)\llcorner(\R^n\times\R_-).
\]
For $ w $ and $ \mu $ as above, let $ V(w) $ and $ V(\mu) $ be the maximal subspaces such that $ w $ and $ \mu $ are backward invariant with respect to $ V(w) $ and $ V(\mu) $, respectively. Define $ \ud(w):=\dim V(w) $ and $ \ud(\mu):=\dim V(\mu) $.

\begin{defn}[Spatial $ k $-symmetry]\label{spatialsym}
For $ k\in\Z\cap[0,n] $, a measurable map $ w:\R^n\times\R\to\cN $ (or a Radon measure $ \mu\in\M(\R^n\times\R) $) is spatially $ k $-symmetric if it is backward self-similar and $ \ud(w)\geq k $ (or $ \ud(\mu)\geq k $).
\end{defn}

To capture the behavior in the time direction, we have to distinguish between the following three cases.

\begin{defn}[Static, quasi-static, and shrinking]
A measurable map $ w $ or a Radon measure $ \mu $ is static if it is independent of $ t\in\R $, quasi-static if it is independent of $ t $ up to some $ T\in[0,+\ift) $, and shrinking otherwise.
\end{defn}

Define $ \rD(w)=\ud(w)+2 $ if $ w $ is static and $ \rD(w)=\ud(w) $ otherwise; similarly for $ \mu $. 

\begin{defn}[Space-time $ k $-symmetry]\label{spacetimesym}
For $ k\in\Z\cap[0,n+2] $, a measurable map $ w $ (or a Radon measure $ \mu $) is space-time $ k $-symmetric if it is backward self-similar and $ \rD(w)\geq k $ (or $ \rD(\mu)\geq k $).
\end{defn}

Having defined $\rD(w)$ and $\rD(\mu)$, we can proceed by classifying points $X\in\text{sing}(u)$ according to the number of symmetries of tangent functions and tangent measures. More precisely, we can introduce a stratification of the singular set as follows.

For $ r>0 $, $ X_0=(x_0,t_0)\in\R^n\times\R, k\in\Z\cap[0,n+2] $ and $ u\in H_{\Lda}(P_r(X_0),\cN) $, define the space-time $ k $-strata as
\begin{align*}
S^k(u)&:=\{X\in P_r(X_0): \rD(w)\leq k~\text{for each tangent flow of }u\text{ at }X\},\\
\Sg^k(u)&:=\{X\in P_r(X_0):\rD(\mu)\leq k~\text{for each tangent measure of }u\text{ at }X\}.
\end{align*}
For $ k\in\Z\cap[0,n] $ and $ t\in(t_0-r^2,t_0+r^2) $, we can also define the spatial $ k $-strata $ \Sg^k(u,t) $ at time $ t $ based on tangent measures. Similarly, we define $ S^k(u,t) $ based on tangent flows as 
\[
\begin{aligned}
S^k(u,t):&=\{x\in B_r(x_0): \ud(w)\leq k~\text{for each tangent flow of }u\text{ at }(x,t)\},\\
\Sg^k(u,t):&=\{x\in B_r(x_0):\ud(\mu)\leq k~\text{for each tangent measure of }u\text{ at }(x,t)\}.
\end{aligned}
\]

By definition, it is easy to check that
\begin{align*}
&S^0(u)\subset S^1(u)\subset...\subset S^{n+2}(u),\\
&\Sg^0(u)\subset\Sg^1(u)\subset...\subset \Sg^{n+2}(u),
\end{align*}
and for any $ t\in(t_0-r^2,t_0+r^2) $,
\begin{gather*}
S^0(u,t)\times\{t\}\subset S^1(u,t)\times\{t\}\subset...\subset S^{n}(u,t)\times\{t\},\\
\Sg^0(u,t)\times\{t\}\subset \Sg^1(u,t)\times\{t\}\subset...\subset\Sg^{n}(u,t)\times\{t\},
\end{gather*}
where for $ E\subset\R^n\times\R $, we have the convention $
E\cap\{t\}:=\{Y=(y,s)\in E:s=t\} $. Due to the abstract stratification theorems given by \cite{Whi97}, we have
\be
\begin{gathered}
\dim_{\cP}(S^k(u)\cup\Sg^k(u))\leq k,\\
\dim_{\HH}(S^k(u,t)\cup\Sg^k(u,t))\leq k\text{ for any }t\in(t_0-r^2,t_0+r^2),\\
\dim_{\HH}((S^k(u)\cap\{t\})\cup(\Sg^k(u)\cap\{t\}))\leq k-2\text{ for a.e. }t\in(t_0-r^2,t_0+r^2).
\end{gathered}\label{StratificaionResults}
\ee
The proof of the above estimates follows from dimensional reduction arguments in the vein of Federer and Almgren. 

In a recent progress \cite{CHN15}, Cheeger, Haslhofer and Naber established that for a suitable solution $ u $ of \eqref{heatflow}, the parabolic Minkowski dimension of $ S^n(u) $ is bounded above by $ n $ under some special assumptions of the target manifold. Although their analysis focuses on tangent flows, their techniques extend naturally to the tangent measures regime. Their proofs are based on the quantitative stratification theory introduced by Cheeger and Naber in \cite{CN13a}. Due to some technical reasons, their methods incurred an $ \va $-loss in the estimates of Minkowski contents. Later, Naber and Valtorta \cite{NV17} refined this framework, introducing a Reifenberg-type theory to produce sharp bounds and optimal rectifiability (see \cite[Chapter 3]{Sim83} for definitions) for each stratum for harmonic maps.

\subsection{Main results, difficulties and strategies}
Inspired by these developments, we will explore spatial stratification of suitable  harmonic map flows using tangent measures.
The first main result in this paper is the following.
\begin{thm}\label{MainTheorem}
Assume $ r>0 , X_0=(x_0,t_0)\in\R^n\times\R $ and $ u:P_{4r}(X_0)\to\cN $ is a suitable solution of the harmonic map flow \eqref{heatflow}. Then for any $ t\in(t_0-r^2,t_0+r^2) $, we have the inclusion 
\be
\Sg^0(u,t)\times\{t\}\subset \Sg^1(u,t)\times\{t\}\subset...\subset\Sg^{n-2}(u,t)\times\{t\}=\sing(u)\cap\{t\}.\label{Topchain}
\ee
Moreover, the following properties hold.
\begin{enumerate}[label=$(\theenumi)$]
\item For any $ k\in\Z\cap[1,n-2] $ and $ t\in(t_0-r^2,t_0+r^2) $, $ \Sg^k(u,t) $ is $ k $-rectifiable. In particular, $ \sing(u)\cap\{t\} $ is $ (n-2) $-rectifiable.
\item Assume  $ \Lda>0 $ and $u\in H_{\Lda}(P_1,\cN)$.  There exists $ C>0 $ depending only on $ \Lda,n $, and $ \cN $ such that
\be
\op{Min}_r^{n-2}(\sing(u)\cap(B_1\times\{t\}))\leq C.\label{MinkowskiContentSgn2}
\ee
In particular, we have
\[
\dim_{\op{Min}}(\sing(u)\cap\{t\})\leq n-2.
\]
\end{enumerate}
Here for $ r>0 $, $ \op{Min}_r^{n-2}(\cdot) $ is the $ n-2 $-dimensional Minkowski $ r $-content $($see \cite[Definition 2.12]{NV17}$)$ and $ \dim_{\op{Min}} $ is the Minkowski dimension $($see \cite[Definition 2.14]{NV17}$)$.
\end{thm}

\begin{rem}
The methods presented in this paper for addressing harmonic map flows are universally applicable across different models. For example, consider the Fujita equation
\[
\pa_tu-\Delta u=|u|^{p-1}u
\]
for the supercritical case where $ p>\f{n+2}{n-2} $. We believe that the techniques and arguments used here can, with minor adjustments, yield results parallel to those found in \cite{FWZ24}. Moreover, our approach applies to other systems, such as mean-curvature flows and semilinear parabolic equations with singular nonlinearities. We recommend references \cite{CHN13,CM16,WY24, WZ24} for further exploration.
\end{rem}

The next main theorem presents an improved regularity result for suitable solutions based on additional assumptions regarding the target manifold $ \cN $. To clarify the context, we first define the necessary concepts.

\begin{defn}[Harmonic spheres and Quasi-harmonic spheres]\label{S2sphere}
For $ k\in\Z\cap[1,n-1] $, a harmonic $ k $-sphere on $ \cN $ is a non-constant smooth map $ \psi:\Ss^k\to\cN $ that is a critical point of 
\[
\int_{\Ss^k}|\na_{\Ss^k}\psi|^2\ud x.
\]
A quasi-harmonic $ k $-sphere on $ \cN $ is a non-constant smooth map $ \psi:\R^k\to\cN $ that is a critical point of 
\[
\int_{\R^k}|\na\psi|^2\exp\(-\f{|x|^2}{4}\)\ud x. \] 
We say $ \cN $ does not admit harmonic $ k $-spheres if there is no harmonic $ k $-sphere on $ \cN $, and similarly for quasi-harmonic $ k $-spheres.
\end{defn}

Given the above definitions, we now present the following regularity estimate.

\begin{thm}\label{ImprovementRegularity}
Let $ \Lda>0 $ and $ u\in H_{\Lda}(P_2,\cN) $. Assume that $ \cN $ admits neither harmonic $ 2 $-spheres nor quasi-harmonic $ 2 $-spheres. Then there exists $ C>0 $ that depends only on $ \Lda,n $, and $ \cN $ such that
\be
\sup_{t\in(-1,1)}\|\na u(\cdot,t)\|_{L^{3,\ift}(B_1)}\leq C.\label{LorentzProperty}
\ee
Here, for $ p\in[1,\ift) $, a measurable set $ E\subset\R^n $, and a function $ f:E\to\R $, we have
\[
\|f\|_{L^{p,\ift}(E)}:=\sup_{t>0}t[\cL^n(\{x\in E:|f(x)|>t\})]^{\f{1}{p}}<+\ift.
\]
\end{thm}
\begin{rem}
The assumptions that $ \cN $ admits neither harmonic nor quasi-harmonic $ 2 $-spheres were also used in \cite{LW99}. They are essential to our proofs because they eliminate certain specific tangent flows and measures. Specifically, the absence of harmonic $ 2 $-spheres ensures that defect measures vanish in the limit, while the disappearance of quasi-harmonic $ 2 $-spheres prevents non-trivial spatially $ (n-2) $-symmetric tangent flows.

In the static setting, similar results are established in \cite{NV17} for Dirichlet energy minimizers. Since it is not clear how to define minimizers for harmonic map flows, we rely on these additional conditions on $ \cN $ to achieve the regularity estimate. 
\end{rem}

\subsubsection{Challenges for our model}
The proofs of the main results are based on the quantitative stratification theory developed in \cite{CN13a, CN13b}, together with the powerful Reifenberg-rectifiable methods introduced by Naber and Valtorta in their seminal work \cite{NV17, NV18}. 

However, the methods mentioned above are not directly applicable to harmonic map flows due to the dynamic nature of flows. Although our study focuses exclusively on spatial stratification, the evolving behavior of flows continues to complicate the analysis. We list several main difficulties as follows.

\emph{1. Lack of nesting property of integral domains.} In the case that $ \psi:\om\to\cN $ is a stationary harmonic map, the density associated with the Dirichlet energy
\[
\Theta(\psi,x,r):=r^{2-n}\int_{B_r(x)}|\na\psi|^2\ud x,\quad B_r(x)\subset\subset\om,
\]
satisfies a monotonicity formula. This monotonicity formula  implies
\[ \Theta(\psi,y,r)\leq C(n)\Theta(\psi,x,2r)\quad\text{for}~y\in B_r(x).
\]

In contrast, harmonic map flows disrupt this simplicity. To illustrate, consider a classical solution $ u:\R^n\times\R\to\cN $ of  \eqref{heatflow}. For a fixed point $ X_0=(x_0,t_0)\in\R^n\times\R $, the  energy density corresponding to $ \Theta(\cdot,\cdot,\cdot) $ is 
\[
E(u,X_0,\rho):=\rho^{-n}\int_{P_{\rho}(X_0)}|\na u|^2\ud x\ud t,
\]
where $ \rho>0 $. Unlike the case of stationary harmonic maps, this quantity lacks a monotonicity formula. We define the backward heat kernel centered at $ X_0 $ as
\[
G_{X_0}(x,t):=\f{1}{(4\pi(t_0-t))^{\f{n}{2}}}\exp\(-\f{|x-x_0|^2}{4(t_0-t)}\),
\]
where $ x\in\R^n $ and $ t<t_0 $. Additionally, for $ \rho>0 $, we  set
\begin{align*}
S_{\rho}(X_0)&:=\{(x,t)\in\R^n\times\R:t=t_0-\rho^2\},\\
T_{\rho}(X_0)&:=\{(x,t)\in\R^n\times\R:t_0-4\rho^2\leq t\leq t_0-\rho^2\}.
\end{align*} 
In \cite{Str88}, Struwe proved that monotonicity holds for alternative densities
\begin{align}
\Psi^0(u,X_0,\rho)&:=\f{\rho^2}{2}\int_{S_{\rho}(X_0)}|\na u|^2(x,t)G_{X_0}(x,t)\ud x,\label{Globaldensity1}\\
\Phi^0(u,X_0,\rho)&:=\f{1}{2}\int_{T_{\rho}(X_0)}|\na u|^2(x,t)G_{X_0}(x,t)\ud x\ud t.\label{Globaldensity2}
\end{align}
As $ \rho $ varies, $ S_{\rho}(X_0) $ forms distinct, non-intersecting slices, and for $ \rho_1\neq\rho_2 $, the symmetric difference \[
T_{\rho_1}(X_0)\Delta T_{\rho_2}(X_0):=(T_{\rho_1}(X_0)\backslash T_{\rho_2}(X_0))\cup(T_{\rho_2}(X_0)\backslash T_{\rho_1}(X_0)) 
\]
is non-empty. Therefore, the integral domains $ S_{\rho}(X_0) $ and $ T_{\rho}(X_0) $ lack the nesting properties of balls or parabolic regions.  These features pose significant obstacles to adapting the arguments from \cite{NV17}.

\emph{2. The lack of the unique continuation property.} Weak solutions obtained as limits of harmonic map flows lack the unique continuation property, a critical tool in \cite{NV18}. This further complicates the analysis. Also, when addressing stratification via tangent measures for a finer characterization of singular sets, we encounter additional difficulties because tangent measures lack unique continuation.

\emph{3. Incompatibility with classical quantitative stratification.} In this paper, we primarily focus on the localized version of the problem  \eqref{heatflow}. Hence, we need to adapt the global densities defined in \eqref{Globaldensity1} and \eqref{Globaldensity2} to localized forms. Following  \cite{LW02b}, we first choose a cut-off function $ \vp\in C_0^{\ift}(B_2) $ such that $ 0\leq\vp\leq 1 $, with $ \vp\equiv 1 $ on $ B_1 $ and $ |\na\vp|\leq 10 $. Suppose $ X_0=(x_0,t_0)\in P_4 $, $ \rho\in(0,2) $, and $ u:P_8\to\cN $ is a suitable solution, we define the localized densities as
\begin{align}
\Psi(u,X_0,\rho)&:=\f{\rho^2}{2}\int_{S_{\rho}(X_0)}\vp_{x_0}^2(x)|\na u(x,t)|^2G_{X_0}(x,t)\ud x,\label{EnergyFunctional2}\\
\Phi(u,X_0,\rho)&:=\f{1}{2}\int_{T_{\rho}(X_0)}\vp_{x_0}^2(x)|\na u(x,t)|^2G_{X_0}(x,t)\ud x\ud t,\label{EnergyFunctional3}
\end{align}
where $ \vp_{x_0}(x)=\vp(x-x_0) $. 

Unlike \eqref{Globaldensity1} and \eqref{Globaldensity2}, these localized densities do not exhibit scaling invariance or a strict monotonicity formula (see Proposition \ref{Monotonicity}). It requires us to establish properties at sufficiently small scales. In contrast, the classical definition of quantitative stratification, as introduced in \cite{CN13a}, does not impose such a limitation on the scale, rendering it incompatible with our framework.

\subsubsection{Our strategies to overcome difficulties}

To address the first two challenges, we draw inspiration from \cite{NV17} and \cite{NV18} to fit our purposes. Specifically, we look to \cite{HSV19} by Hirsch, Stuvard, and Valtorta, which examines energy-minimizing harmonic maps with multiple values, and \cite{EE19} by Edelen and Engelstein, which focuses on nonlinear free boundary problems. In \cite{HSV19}, the authors used cut-off functions to bypass the reliance on the unique continuation and nesting properties of integral domains to define monotone densities. Meanwhile, \cite{EE19} used compactness arguments to avoid these issues altogether.

In this paper, the application of cut-off functions is less straightforward due to the time dimension, making it hard to use arguments in \cite{HSV19} to the evolutionary problem. Fortunately, with careful adjustments, the approach in \cite{EE19} suits the flow setting very well. A key step in this adaptation is to establish the parabolic version of the ``$ L^2 $-best estimates" (Proposition \ref{L2BestProp}), which associates ``displacements" with differences in the ``monotone" localized energy densities across two scales. To our knowledge, this insight has not previously been used in harmonic map flow studies. Regarding the third difficulty, we refine the definition of quantitative stratification by introducing more subtle restrictions on the scales. We also establish inclusion results for quantitative strata at different scales, connecting theorems under this new definition to the classical one. It represents another novelty of our paper.

We organize this paper as follows. In Section \ref{Preliminaries}, we gather some key concepts related to suitable solutions, including the monotonicity formula, partial regularity theory, compactness properties, and blow-up results. In Section \ref{QuantitativeStratification}, we introduce the fundamentals of quantitative spatial stratification. In Section \ref{CoveringLemmas}, we derive several covering results which are vital to proving the main theorems. In Section \ref{Reifenbergtype}, we recall the Reifenberg-type theorems from geometric measure theory. In Section \ref{L2best}, we present the $ L^2 $-best approximation results, highlighting their connections to Reifenberg-type theorems and the monotonicity formula. In Section \ref{DichotomySection}, we establish several useful preliminary results, which are then used in Section \ref{ProofofCoverSection} to complete the proof of the key covering lemma stated in Section \ref{CoveringLemmas}. Finally, in Section \ref{ProofofMainTheorems}, we use the covering lemmas of Section \ref{CoveringLemmas} to finalize the proofs of the main theorems.
\subsection{Notations and conventions} 
\begin{itemize}
\item Throughout this paper, $ C $ will be a positive constant that can change from line to line. To highlight the dependence on parameters $ a,b,... $, we may write $ C(a,b,...) $.
\item For $ u,v\in \R^d $, we use $\<u,v\>_{\R^d} $ or $u\cdot v$ to denote the inner product.
\item For $ k\in\Z\cap[1,n] $, the Grassmannian $ \bG(n,k) $ represents all $ k $-dimensional subspaces of $ \R^n $, while $ \bA(n,k) $ denotes all $ k $-dimensional affine subspaces of $ \R^n $. For a sequence $ \{V_i\}\subset\bG(n,k) $, we say 
\[
V_i\to V\in\bG(n,k)\quad\Longleftrightarrow\quad d_G(V_i,V)\to 0,
\]
where $ d_G(\cdot,\cdot) $ is the Grassmannian metric. Similarly, for $ \{L_i\}\subset\bA(n,k) $, we have 
\[
 L_i=x_i+V_i\to L=x+V\in\bA(n,k)\quad\Longleftrightarrow\quad\{V_i\}\subset\bG(n,k),\,x_i\to x\text{ and }V_i\to V.
\]
\item For a $ k $-dimensional subspace $ V=\op{span}\{v_i\}_{i=1}^k $ with orthonormal basis $ \{v_i\}_{i=1}^k $ and $ u\in H^1(\R^n) $, we define 
\[
|V\cdot\na u|^2:=\sum_{i=1}^k|v_i\cdot\na u|^2.
\] 
\item For $ \al\in[0,n] $, we use $ \HH^{\al}(\cdot) $ to denote the $ \al $-dimensional Hausdorff measure.
\item For $ \al\in[0,n+2] $, $ \cP^{\al}(\cdot) $ denotes the Hausdorff measure of dimension $ \al $ under the parabolic metric 
\[
\dist_{\cP}(X,Y):=(|x-y|^2+|t-s|)^{\f{1}{2}},
\]
where $ X=(x,t),Y=(y,s)\in\R^n\times\R $.  
\item Let $ \mu\in\M(U) $ and $ U\subset\R^k $. For any $ \mu $-measurable set $ E\subset U $ with $ \mu(E)>0 $, we define
\[
\dashint_Ef\ud\mu=\f{1}{\mu(E)}\int_Ef\ud\mu\quad\text{for any }\mu\text{-measurable }f:E\to\R. 
\]
\end{itemize}

\section{Preliminaries}\label{Preliminaries}

\subsection{Localized energy inequality}

In this subsection, we give the localized energy inequality for suitable solutions of harmonic map flows.

\begin{lem}[Localized energy inequality]\label{LocalEnergy}
Let $ r>0 $ and $ X_0=(x_0,t_0)\in P_2 $ be a point such that $ P_r(X_0)\subset P_4 $. Assume $ u:P_4\to\cN $ is a suitable solution of \eqref{heatflow}, then
\be
\begin{aligned}
&\int_{B_r(x_0)\times[s,t]}|\pa_tu|^2\phi^2\ud x\ud\tau+\int_{B_r(x_0)}|\na u(\cdot,t)|^2\phi^2\ud x\\
\leq& \int_{B_r(x_0)}|\nabla u(\cdot,s)|^2\phi^2\ud x+4\int_s^t\int_{B_r(x_0)}|\na u|^2|\na\phi|^2\ud x\ud\tau
\end{aligned}\label{LocalEnergyEq}
\ee
for any $ \phi\in C_0^{\ift}(B_r(x_0)) $ and a.e. $ t_0-r^2<s\leq t<t_0+r^2 $.
\end{lem}

\begin{proof}
Let $\{\xi_{k}\}$ be a sequence of smooth functions that approximate the characteristic function $\chi_{[s,t]}$. Testing \eqref{LocalizedEnergy} against $\xi_{k}\phi$ and passing to the limit $k\to+\infty$ yields \eqref{LocalEnergyEq}. We refer the readers to \cite[Proposition 8]{Fel94} for more details.
\end{proof}

Choosing some specific $ \phi $ in \eqref{LocalEnergyEq} and using average arguments, we have the following result.

\begin{cor}\label{LocalImprovedEstimates}
Let $ \Lda>0 $, $ r>0 $, and $ X_0=(x_0,t_0)\in\R^n\times\R $. If $ u:P_{2r}(X_0)\to\cN $ is a suitable solution of \eqref{heatflow}, then there exists $ C>0 $, depending only on $ n $ such that for a.e. $ t\in(t_0-r^2,t_0+r^2) $,
\[
\int_{B_r(x_0)}|\na u(\cdot,t)|^2\ud x+\int_{P_r(X_0)}|\pa_{\tau}u|^2\ud x\ud\tau\leq \f{C}{r^2}\int_{P_{2r}(X_0)}|\na u|^2\ud x\ud t.
\]
\end{cor}

\subsection{Monotonicity formula} Recall the localized densities $ \Psi(\cdot,\cdot,\cdot) $ and $ \Phi(\cdot,\cdot,\cdot) $ as in \eqref{EnergyFunctional2} and \eqref{EnergyFunctional3}. For $ 0<r<R $, we set
\be
\cW(u,X_0,R,r):=\Phi(u,X_0,R)-\Phi(u,X_0,r).\label{Difference}
\ee

The next proposition is the aforementioned localized energy monotonicity formula of Struwe's \cite[Lemma 3.2]{Str88} type for suitable solutions.
\begin{prop}[Monotonicity formula]\label{Monotonicity}
 Assume $\Lambda>0$, $ X_0=(x_0,t_0)\in P_4 $ and  $ u\in H_{\Lda}(P_8,\cN) $. Then 
for a.e. $ 0<r\leq R<2 $,
\begin{align*}
&\Psi(u,X_0,R)-\Psi(u,X_0,r)+C_1(R-r)\\
&\quad\quad\geq C_2\int_{t_0-R^2}^{t_0-r^2}\(\int_{B_1(x_0)}\vp_{x_0}^2\f{|(x-x_0)\cdot\na u+2(t-t_0)\pa_tu|^2}{|t_0-t|}G_{X_0}(x,t)\ud x\)\ud t,
\end{align*}
and for a.e. $ 0<r\leq R<2 $,
\begin{align*}
&\Phi(u,X_0,R)\exp(C_1(R-r))-\Phi(u,X_0,r)+C_1(R-r)\\
&\quad\quad\geq C_2\int_r^R\(\int_{T_{\rho}(X_0)}\vp_{x_0}^2\f{|(x-x_0)\cdot\na u+2(t-t_0)\pa_tu|^2}{|t_0-t|}G_{X_0}(x,t)\ud x\ud t\)\f{\ud\rho}{\rho},
\end{align*}
where $ C_1,C_2>0 $ depend only on $ \Lda $ and $ n $.
\end{prop}
\begin{proof}
The results follow from  \cite[Lemma 2.4]{Wan08}.
\end{proof}
As a consequence of Proposition \ref{Monotonicity}, we have the following result.

\begin{cor}\label{corselfsimilar}
Let $ u:P_r(X_0)\to\cN $ be a suitable solution of \eqref{heatflow}, $ X\in P_r(X_0) $, and let $\{r_i\}$ be a sequence with $ r_i\to 0 $ as $ i\to\infty $. Assume
\begin{gather*}
w_i:=T_{X,r_i}u\rightharpoonup w \quad \text{weakly in } H_{\loc}^1(\R^n\times\R,\cN),\\
\frac{1}{2}|\nabla w_i|^2\ud x\ud t \rightharpoonup^* \mu:=\frac{1}{2}|\nabla w|^2\ud x\ud t+\nu \quad \text{in } \M(\R^n\times\R),
\end{gather*}
Then $ w $, $ \mu $, and $ \nu $ are backward self-similar.
\end{cor}

\begin{proof}
The argument here is inspired by the proof of a related result in the Ginzburg–Landau setting (see \cite[Lemma 3.2]{LW02b}). Without loss of generality, we assume $ X=0^{n,1} $. Using Proposition \ref{Monotonicity}, the limit $ \lim_{\rho\to 0^+}\Phi(u,0^{n,1},\rho) $
exists. Hence, for a.e. $ -\infty<r<R<+\infty$, we have
\begin{align*}
0
&\leq \limsup_{i\to+\infty}\int_r^R\left(\int_{T_{\rho}(0^{n,1})}\vp_{0^n}^2(r_i y)\frac{|y\cdot\nabla w_i+2s\partial_s w_i|^2}{|s|}G_{0^{n,1}}(y,s)\ud y\ud s\right)\frac{\ud\rho}{\rho}\\
&\leq \limsup_{i\to+\infty} C(\Lambda,n,\cN,R,r)(\Phi(u,0^{n,1},Rr_i)-\Phi(u,0^{n,1},rr_i)+r_i)=0.
\end{align*}
It follows that
\be
\lim_{i\to+\infty}\int_U |y\cdot\nabla w_i+2s\partial_s w_i|^2=0 \label{Uintegral}
\ee
for any open set $ U $ with $ \overline{U}\subset \R^n\times\R $. Passing to the limit and using the weak convergence of $ w_i $, we deduce that $ w $ is backward self-similar.

It remains to prove that $ \mu $ is backward self-similar. Let $ \phi\in C_0^\infty(\R^n\times\R_-) $ and define
\[
J(\lambda):=\int_{\R^n\times\R_-}\phi(\lambda y,\lambda^2 s)G_{0^{n,1}}(y,s)\ud\mu(y,s).
\]
Set $ G=G_{0^{n,1}}, \phi_\lambda(y,s)=\phi(\lambda y,\lambda^2 s) $ and 
\[
J_i(\lambda):=\f{1}{2}\int_{\R^n\times\R_-}\phi_\lambda G|\nabla w_i|^2\ud y\ud s.
\]
By the convergence of measures, $ J_i'(\lambda)\to J'(\lambda) $ as $ i\to+\infty $. 

We now derive an identity for $ J_i'(\lambda) $. Testing \eqref{LocalizedEnergy} with $ \theta=s\phi_\lambda G $,\footnote{\eqref{LocalizedEnergy} is assumed as an equality, slightly stronger than in \cite{Fel94}, only to recover the exact identity \eqref{testuse11}.} we have
\be
\begin{aligned}
0
&=\int_{\R^n\times\R_-}\left[\left(\frac{n}{2}-1\right)\phi_\lambda - s\partial_s\phi_\lambda + \frac{|y|^2}{4s}\phi_\lambda\right]G|\nabla w_i|^2 \\
&\quad +\int_{\R^n\times\R_-}(2sG\partial_s w_i(\nabla w_i\cdot\nabla\phi_\lambda)+2s\phi_\lambda G|\partial_s w_i|^2
+\phi_\lambda G\partial_s w_i(y\cdot\nabla w_i)).
\end{aligned}\label{testuse11}
\ee
Next, testing \eqref{StationaryCondition2} with $ \xi=y\phi_\lambda G $, we obtain
\begin{align*}
0&=\int_{\R^n\times\R_-}\left((n-2)\phi_\lambda + y\cdot\nabla\phi_\lambda + \frac{|y|^2}{2s}\phi_\lambda\right)G|\nabla w_i|^2 \\
&\quad-\int_{\R^n\times\R_-}\(2\phi_\lambda G\partial_s w_i(y\cdot\nabla w_i)+2G(y\cdot\nabla w_i)(\nabla\phi_\lambda\cdot\nabla w_i)
+\frac{\phi_\lambda G}{s}|y\cdot\nabla w_i|^2\).
\end{align*}
Combining the above two identities, we obtain
\begin{align*}
\int_{\R^n\times\R_-}(y\cdot\nabla\phi_\lambda+2s\partial_s\phi_\lambda)G|\nabla w_i|^2
&=2\int_{\R^n\times\R_-}G(\nabla\phi_\lambda\cdot\nabla w_i)(y\cdot\nabla w_i+2s\partial_s w_i)\\
&\quad+\int_{\R^n\times\R_-}\frac{\phi_\lambda G}{s}|y\cdot\nabla w_i+2s\partial_s w_i|^2.
\end{align*}
A direct computation shows that
\begin{align*}
\lambda J_i'(\lambda)
&=\int_{\R^n\times\R_-}G(\nabla\phi_\lambda\cdot\nabla w_i)(y\cdot\nabla w_i+2s\partial_s w_i)+\int_{\R^n\times\R_-}\frac{\phi_\lambda G}{2s}|y\cdot\nabla w_i+2s\partial_s w_i|^2.
\end{align*}
Using \eqref{Uintegral} and Cauchy's inequality, we take the limit $ i\to+\infty $ and conclude that $ J'(\lambda)=0 $ for all $ \lambda>0 $. This proves that $ \mu $ is backward self-similar. The same property holds for $ \nu $, which completes the proof.
\end{proof}

The following lemma is the key observation in the original dimension reduction argument. It is frequently applied throughout the remainder of the paper.

\begin{lem}\label{ConeSplittingCor}
For $ k\in\Z\cap[0,n-1] $, let $ X=(x,0)\in\R^n\times\R $ with $ x\neq 0^n $, and $ V\in\bG(n,k) $. If a measurable map $ u:\R^n\times\R\to\R^d $ and a Radon measure $ \mu\in\M(\R^n\times\R) $ are backward self-similar at $ 0^{n,1},X $ and backward invariant with respect to $ V $, then both $ u $ and $ \mu $ are backward invariant with respect to $ \op{span}\{x,V\}\in\bG(n,k+1) $.
\end{lem}
For $ r\in(0,2) $, we define the backward energy density as
\[
E_-(u,X_0,r):=r^{-n}\int_{P_r^-(X_0)}|\na u|^2\ud x\ud t,
\]
 where $ P_r^-(X_0):=B_r(x_0)\times(t_0-r^2,t_0) $. The next lemma builds the relationship between the backward energy density and the energy density defined on a full parabolic ball.

\begin{lem}\label{ComparisonEandPhi}
Let $ \Lda>0 $, $ r\in(0,1) $ and $ X_0=(x_0,t_0)\in P_4 $. Assume that $ u\in H_{\Lda}(P_8,\cN) $ is a suitable solution of \eqref{heatflow}. There is $ \delta_0\in(0,\f{1}{4}) $, depending only on $ n $ and $\Lambda
$ such that if $ \delta\in(0,\delta_0) $, then
\be
E(u,X_0,\delta r)\leq C(\delta^{-n}E_-(u,X_0,2r)+\delta r),\label{EuX0Bound}
\ee
where $ C>0 $ depends only on $ \Lda $ and $ n $.
\end{lem}
\begin{proof}
Without loss of generality, we may assume $ X_0= 0^{n,1} $. If $ \delta\in(0,\f{1}{4})$ and $X=(x,t)\in P_{\delta r} $, we have $
G_{(0^n,2\delta^2r^2)}(x,t)\geq c(n)(\delta r)^{-n} $. Then, Proposition \ref{Monotonicity} implies 
\be
E(u,0^{n,1},\delta r)\leq C(n) \Phi(u,(0^n,2\delta^2r^2),\delta r)\leq C(n)\Big[\Phi(u,(0^n,2\delta^2r^2),2\delta r)+C(\Lda,n)\delta r\Big]\label{NablauLess}
\ee
for any $ r\in(0,1) $. On the other hand,
\[
\begin{aligned}
\Phi(u,(0^n,2\delta^2r^2),2\delta r)&\leq C \int_{B_1\times[-14\delta^2r^2,-2\delta^2r^2]\cap P_r}|\na u|^2G_{(0^n,2\delta^2r^2)}\ud x\ud t\\
&\quad+C\int_{B_1\times[-14\delta^2r^2,-2\delta^2r^2]\backslash P_r}|\na u|^2G_{(0^n,2\delta^2r^2)}\ud x\ud t\\
&\leq C(\Lda,n)\left[\delta^{-n}E_-(u,0^{n,1},r)+\delta^{-n}\exp\(-\f{1}{56\delta^2}\)\int_{-14\delta^2r^2}^{-2\delta^2r^2}\int_{B_1}|\na u|^2\ud x\ud t\right],\label{Phiu00}
\end{aligned}
\]
where the second inequality follows from the property
\[
G_{(0^n,2\delta^2r^2)}(x,t)\leq C(n)\delta^{-n}\exp\(-\f{1}{56\delta^2}\)
\]
for any $ (x,t)\in\R^n\times[-14\delta^2r^2,-2\delta^2r^2]\backslash P_r $. Provided $ \delta\in(0,\delta_0) $ with $ \delta_0= \delta_0(n)>0 $ being sufficiently small, Corollary \ref{LocalImprovedEstimates} together with \eqref{NablauLess} produces \eqref{EuX0Bound}.
\end{proof}

\subsection{Partial regularity results}

In this subsection, we review the partial regularity theory for suitable solutions of harmonic map flows.

\begin{prop}\label{PartialRegularity}
Assume $ r\in(0,1), X_0\in P_4 $ and $ u:P_8\to\cN $ is a suitable solution of \eqref{heatflow}. There exist $ \va>0 $, depending only on $ n $ and $ \cN $ such that if $ E(u,X_0,2r)\leq\va^2 $, then $ u\in C^{\ift}(P_r(X_0),\cN) $, and  
\[
\sup_{P_r(X_0)}(r^2|\pa_tu|+r|\na u|)\leq C(\va),
\]
where $ C(\va)>0 $ depends only on $ \va,n $, and $ \cN $ such that $ \lim_{\va\to 0^+}C(\va)=0 $.
\end{prop}
\begin{proof}
This follows from Lemma \ref{ComparisonEandPhi}, \cite[Lemma 5.1]{Liu03}, and iteration techniques.
\end{proof}

As a direct consequence of Lemma \ref{ComparisonEandPhi} and Proposition \ref{PartialRegularity}, we can obtain the following.
\begin{cor}\label{PartialRegularityCor}
Assume $ X_0\in P_4$ and $ u\in H_{\Lda}(P_8,\cN) $. There exist $ \va,\delta>0 $ depending only on $ \Lda,n,\cN $ such that if $ E_-(u,X_0,2r)\leq\va^2 $, then $ u\in C^{\ift}(P_{\delta r}(X_0),\cN) $ and
\[
\sup_{P_{\delta r}(X_0)}((\delta r)^2|\pa_tu|+(\delta r)|\na u|)\leq C(\va),
\]
where $ C(\va)>0 $ depends only on $ \va,n $, and $ \cN $ such that $ \lim_{\va\to 0^+}C(\va)=0 $.
\end{cor}

\begin{defn}[Regularity scale]\label{RegularScale}
For a suitable solution $ u:P_8\to\cN $ of \eqref{heatflow}, we define the regularity scale at a point $ X=(x,t)\in P_4 $ as
\[
r_u(X):=\sup\left\{r\in[0,1]:\sup _{P_r(X)}(r^2|\pa_tu|+r|\na u|)\leq 1\right\}.
\]
Here, we use the convention that $r_u(X)= 0$ if $X$ is a singular point.
\end{defn}

Using Corollary \ref{PartialRegularityCor}, we can derive the following estimate on regularity scales.

\begin{lem}\label{Regularscaleresult}
Given $ X_0\in P_4 $ and $ u\in H_{\Lda}(P_8,\cN) $. There exist $ \va,\delta>0 $, depending only on $ \Lda,n $, and $ \cN $ such that if $ E_-(u,X_0,2r)\leq\va^2 $, then $ r_u(X_0)\geq\delta r $.
\end{lem}

\subsection{Compactness and blow-ups} In this subsection, we review the compactness properties of suitable solutions of \eqref{heatflow} and apply these results to conduct a blow-up analysis.

\begin{lem}\label{UniformBound}
Let $ \Lda>0 $, $ r>0 $ and  $ X_0=(x_0,t_0)\in\R^n\times\R $. Suppose that $ u\in H_{\Lda}(P_{4r}(X_0),\cN) $. Then there exists a constant $ C>0 $, depending only on $ \Lda $ and $ n $ such that
\[
\sup_{\rho\leq r,\,\,X\in P_{2r}(X_0)}\(\rho^{-n}\int_{P_{\rho}(X)}|\na u|^2+\rho^{2-n}\int_{P_{\rho}(X)}|\pa_tu|^2\)\leq C.
\]
\end{lem}
\begin{proof}
It follows directly from Corollary \ref{LocalImprovedEstimates} and Proposition \ref{Monotonicity}. One can also refer to \cite[Lemma 2.2]{Liu03} for more details.
\end{proof}

Next, we present a compactness result for a sequence of suitable solutions of harmonic map flows.

\begin{prop}\label{Compactness}
Assume that $ \Lda>0 $ and $ \{u_i\}\subset H_{\Lda}(P_4,\cN) $. There exist $ u\in H^1(P_4,\cN) $ and Radon measures $ \nu,\eta\in\M(P_4) $ such that up to a subsequence,
\begin{gather*}
u_i\wc u\text{ weakly in }H_{\loc}^1(P_1,\cN),\\
|\pa_tu_i(x,t)|^2\ud x\ud t\wc^*|\pa_tu(x,t)|^2\ud x\ud t+\eta\text{ in }\M(P_1),\\
\f{1}{2}|\na u_i(x,t)|^2\ud x\ud t\wc^*\mu:=\f{1}{2}|\na u(x,t)|^2\ud x\ud t+\nu\text{ in }\M(P_1).
\end{gather*}
 Define the concentration set
\[
\Sg:=\bigcap_{r>0}\left\{X\in P_1:\liminf_{i\to+\ift}E(u_i,X,r)\geq\va^2\right\},
\]
where $ \va>0 $ is the constant given by Proposition \ref{PartialRegularity}. Then the following properties hold.
\begin{enumerate}[label=$(\theenumi)$]
\item $ \Sg $ is relatively closed in $ P_1 $, and for any $ R\in(0,1) $, we have $ \cP^n(\Sg\cap P_R)\leq C $, where $ C>0 $ depends only on $ \Lda,n,\cN $, and $ R $. 
\item $ u\in C^{\ift}(P_1\backslash\Sg,\cN) $, and 
\[
u_i\to u\text{ strongly in }H_{\loc}^1\cap C_{\loc}^{\ift}(P_1\backslash\Sg,\cN).
\]
Additionally, $ u $ is a weak solution of \eqref{heatflow}.
\item $ \sing(u)\cup\supp(\nu)=\Sg $ and $ \supp\eta \subset\Sg $.
\item For any $ R\in(0,1) $, there exists $ C>0 $, depending only on $ \Lda,n,\cN $, and $ R $ such that 
\[
\HH^{n-2}(\Sg\cap(B_R\times\{t\}))\leq C\quad\text{for any }t\in\(-\f{1}{2},\f{1}{2}\).
\]
\item If the target manifold $ \cN $ does not admit harmonic $ 2 $-spheres, then $ \cP^n(\Sg)=0 $. In particular, $ u_i\to u $ strongly in $ H^1(P_1,\cN) $.
\end{enumerate}
\end{prop}
\begin{proof}
For the first three properties, we refer to \cite[Proposition 6.1]{Str88} or \cite[Theorem 3.1]{CS89} for a verbatim argument in the context of the Ginzburg-Landau equation. The fourth property is proved in the main theorem of \cite{Che91}.\footnote{The proof in \cite{Che91} relies on the monotonicity formula, which can be ensured for suitable solutions despite \cite{Che91} assuming smoothness.} The fifth property corresponds to \cite[Theorem C]{LW99} for the Ginzburg-Landau model, adjusted here accordingly.
\end{proof}

The following proposition is concerned about the blow-up limits for suitable solutions of \eqref{heatflow}.

\begin{prop}\label{BlowUplimits}
Assume that $ \{u_i\}\subset H_{\Lda}(P_2,\cN) $, $ \{X_i=(x_i,t_i)\}\subset P_1 $, and $ r_i\to 0^+ $. There exist $ v:\R^n\times\R\to\cN $ with $ \pa_tv,\na v\in L_{\loc}^2(\R^n\times\R) $, and $ \nu\in\M(\R^n\times\R) $ such that up to a subsequence, 
\be
\begin{gathered}
v_i:=T_{X_i,r_i}u_i\wc v\text{ weakly in }H_{\loc}^1(\R^n\times\R,\cN),\\
\f{1}{2}|\na v_i(x,t)|^2\ud x\ud t\wc^*\mu:=\f{1}{2}|\na v(x,t)|^2\ud x\ud t+\nu\text{ in }\M(\R^n\times\R).
\end{gathered}\label{ConvergencePropertyvi}
\ee
Define
\[
\Sg:=\bigcap_{r>0}\left\{X\in \R^n\times\R:\liminf_{i\to+\ift}E(v_i,X,r)\geq\va^2\right\},
\]
where $ \va>0 $ is the constant given by Proposition \ref{PartialRegularity}. Then the following properties hold.
\begin{enumerate}[label=$(\theenumi)$]
\item\label{BlowUplimits1} $ \Sg $ is relatively closed in $ \R^n\times\R $, and for any $ R\in(0,1) $, there exists $ C>0 $, depending only on $ \Lda,n,\cN $, and $ R>0 $ such that $ \cP^n(\Sg\cap P_R)\leq C $.
\item\label{BlowUplimits2} $ v\in C^{\ift}((\R^n\times\R)\backslash\Sg,\cN) $ and
\[
v_i\to v\text{ strongly in }H_{\loc}^1\cap C_{\loc}^{\ift}((\R^n\times\R)\backslash\Sg,\cN).
\]
Additionally, $ v $ is a weak solution of the harmonic map flow \eqref{heatflow}.
\item\label{BlowUplimits3} $ \sing(v)\cup\supp(\nu)=\Sg $.
\item\label{BlowUplimits4} For any $ R>0 $, 
\[
\HH^{n-2}(\Sg\cap(B_R\times\{t\}))\leq C\quad\text{for any }t\in(-R^2,R^2),
\]
where $ C>0 $ depends only on $ \Lda,n,\cN $, and $ R $.
\item\label{BlowUplimits5} If the target manifold $ \cN $ does not admit harmonic $ 2 $-spheres, then $ \cP^n(\Sg)=0 $ and 
\[
v_i\to v\text{ strongly in }H_{\loc}^1(\R^n\times\R,\cN).
\]
\item\label{BlowUplimits6} If for some $ 0<r<R<+\ift $, $ \cW(u_i,X_i,Rr_i,rr_i)\to 0^+ $, then $ v,\nu $, and $ \mu $ are all backward self-similar in $ \R^n\times(-4R^2,-r^2) $. In particular, if $ u_i\equiv u $ and $ X_i\equiv X $ for any $ i\in\Z_+ $, then $ v,\nu $, and $ \mu $ are all backward self-similar in $ \R^n\times\R_- $.
\item\label{BlowUplimits7} If for some $ 0<r<R<+\ift $ and $ k\in\Z\cap[1,n] $ there are $ \{V_i\}\subset\bG(n,k) $ such that $ V_i\to V\in\bG(n,k) $, and
\be
\lim_{i\to+\ift}\(r_i^{-n}\int_{t_i-R^2r_i^2}^{t_i-r^2r_i^2}\int_{B_{r_i}(x_i)}|V_i\cdot\na u_i|^2\ud x\ud t\)=0,\label{ViGradientU}
\ee
then $ v,\nu $, and $ \mu $ are all backward invariant with respect to $ V $ in $ B_1\times(-R^2,-r^2) $.
\item\label{BlowUplimits8} For $ \rho>0 $ and $ X\in\R^n\times\R $, let
\be
\Phi(\mu,X,\rho):=\f{1}{2}\int_{T_{\rho}(X)}G_X(y,s)\ud\mu(y,s).\label{Phimudef}
\ee
For any $ R_2>R_1>0 $ and $ X=(x,t)\in\R^n\times\R $, we have
\be
\begin{aligned}
&C(\Phi(\mu,X,R_2)-\Phi(\mu,X,R_1))\\
&\quad\geq\int_{R_1}^{R_2}\(\int_{T_{\rho}(X)}\f{|(y-x)\cdot\na v+2(s-t)\pa_sv|^2}{|s-t|}G_X(y,s)\ud y\ud s\)\f{\ud\rho}{\rho},
\end{aligned}\label{Phimu}
\ee
where $ C>0 $ depends only on $ \Lda,n $, and $ \cN $.
\end{enumerate}
\end{prop}
\begin{proof}
The existence of $ v $ and the convergence of $ v_i $ follow from Lemma \ref{UniformBound}. The first five properties are inherited from Proposition \ref{Compactness}. 

For the sixth property, applying Proposition \ref{Monotonicity} to $ \{u_i\} $ and using standard scaling, we obtain
\begin{align*}
&\int_r^R\(\int_{T_{\rho}(0^{n,1})}\vp_{0^n}^2(r_iy)\f{|y\cdot\na v_i+2s\pa_sv_i|^2}{|s|}G_{0^{n,1}}(y,s)\ud y\ud s\)\f{\ud\rho}{\rho}\\
&\quad\leq  C(\Lda,n,\cN,R,r)(\cW(u_i,X_i,Rr_i,rr_i)+r_i).
\end{align*}
If $ \cW(u_i,X_i,Rr_i,rr_i)\to 0^+ $, then
\[
\lim_{i\to+\ift}\int_{-4R^2}^{-r^2}\(\int_{\R^n}\phi|y\cdot\na v_i+2s\pa_sv_i|^2\ud y\)\ud s=0
\]
for any $ \phi\in C_0^{\ift}(\R^n\times(-4R^2,-r^2),\R_{\geq 0}) $. Arguing as in the proof of Corollary \ref{corselfsimilar}, we conclude that $ v $, $ \nu $, and $ \mu $ are backward self-similar (see also \cite[Lemma 3.2]{LW02b}).

For the seventh property, the invariance of $ v $ with respect to $ V $ is a consequence of  \eqref{ViGradientU}. To prove $ \mu $ is invariant with respect to $ V $, it is enough to show that for any $ \phi\in C_0^{\ift}(B_1\times(-R^2,-r^2))$ and $ \tau\in V $,
\be
\int_{B_1}\int_{-R^2}^{-r^2}\phi(\cdot+\lda\tau,\cdot)G_{0^{n,1}}(\cdot+\lda \tau,\cdot)\ud\mu(y, s)=\int_{B_1}\int_{-R^2}^{-r^2}\phi G_{0^{n,1}}\ud\mu(y, s)\label{theresult1}
\ee
whenever $ \lda>0 $ with $ \phi(\cdot+\lda\tau,\cdot)\in C_0^{\ift}(B_1\times(-R^2,-r^2)) $. Indeed, the desired result follows from the arbitrariness of $ \phi $ and $ \tau $. Through simple calculations, it holds that
\be
\begin{aligned}
&\f{\ud}{\ud\lda}\(\int_{B_1}\int_{-R^2}^{-r^2}\phi(\cdot+\lda\tau,\cdot)G(\cdot+\lda\tau, \cdot)
\ud\mu(y, s)\)\\
=&\int_{B_1}\int_{-R^2}^{-r^2}\(\f{[(\phi(y+\lda\tau,s)(y+\lda\tau)+2s\na\phi(y+\lda\tau,s))]\cdot\tau}{2s}\)G_{0^{n,1}}(y+\lda\tau,s)\ud\mu(y,s).
\end{aligned}\label{phiGmu}
\ee
Using \eqref{StationaryCondition2} with $ \xi=\tau\phi(\cdot+\lda\tau,\cdot)G_{0^{n,1}}(\cdot+\lda\tau,\cdot) $, we obtain that
\begin{align*}
&\int_{B_1}\int_{-R^2}^{-r^2}\(\f{[(\phi(y+\lda\tau,s)(y+\lda\tau)+2s\na\phi(y+\lda\tau,s))]\cdot\tau}{4s}\)|\na v_i|^2G_{0^{n,1}}(y+\lda\tau,s)\ud y\ud s\\
=&\int_{B_1}\int_{-R^2}^{-r^2}(\tau\cdot\na v_i)(\na\phi(y+\lda\tau,s)\cdot\na v_i)G_{0^{n,1}}(y+\lda\tau,s)\ud y\ud s\\
&+\int_{B_1}\int_{-R^2}^{-r^2}\(\f{((y+\lda\tau)\cdot\na v_i)(\tau\cdot\na v_i)}{2s}\)G_{0^{n,1}}(y+\lda\tau,s)\ud y\ud s\\
&+\int_{B_1}\int_{-R^2}^{-r^2}\pa_tv_i\na v_i\cdot\tau\phi(y+\lda\tau,s)G_{0^{n,1}}(y+\lda\tau,s)\ud y\ud s.
\end{align*}
It follows from \eqref{ViGradientU} that the left-hand side of the above formula converges to $ 0 $ as  $ i\to+\ift $. Consequently, the right-hand side of \eqref{phiGmu} is $ 0 $, and we obtain \eqref{theresult1}. 

It remains to prove the last property. With the help of Proposition \ref{Monotonicity}, for $ R_2>R_1>0 $ and $ X\in\R^n\times\R $, we have
\begin{align*}
&\Phi(u_i,P_{X_i,r_i}^{-1}(X),R_2r_i)\exp(C_1(R_2-R_1)r_i)-\Phi(u_i,P_{X_i,r_i}^{-1}(X),R_1r_i)+C_1(R_2-R_1)r_i\\
&\geq C_2\int_{R_1r_i}^{R_2r_i}\(\int_{T_{\rho}(P_{X_i,r_i}^{-1}(X))}\vp^2(y-y_i)\f{|(y-y_i)\cdot\na u_i+2(s-s_i)\pa_su_i|^2}{|s-s_i|}G_{P_{X_i,r_i}^{-1}(X)}(y,s)\ud y\ud s\)\f{\ud\rho}{\rho}\\
&=C_2\int_{R_1}^{R_2}\(\int_{T_{\rho}(X)}\vp^2(r_i(y-x))\f{|(y-x)\cdot\na v_i+2(s-t)\pa_sv_i|^2}{|s-t|}G_{X}(y,s)\ud y\ud s\)\f{\ud\rho}{\rho},
\end{align*}
where $ y_i=x_i+r_ix $ and $ s_i=t_i+r_i^2t $. Taking $ i\to+\ift $ and using \eqref{ConvergencePropertyvi}, the reasoning is closed.
\end{proof}
\begin{rem}\label{Remarkmubound}
It follows from Lemma \ref{UniformBound} that 
\[
\sup_{\rho>0,\,\,X\in\R^n\times\R}\rho^{-n}\mu(P_{\rho}(X))\leq C,
\]
where $ C>0 $ depends only on $ \Lda,n $, and $ \cN $. Thus, the quantity in \eqref{Phimudef} is well-defined.
\end{rem}

Finally, we apply the above compactness results to build connections between the top space-time strata and the top spatial strata.

\begin{prop}\label{SgnSgn2}
Assume that $ u:P_2\to\cN $ is a suitable solution for \eqref{heatflow}. Then, for any $ t\in(-4,4) $, $ \Sg^n(u)\cap\{t\}=\Sg^{n-2}(u,t)\times\{t\} $.
\end{prop}
\begin{proof}
It suffices to show that \[  \Sg^n(u)\cap\{t_{0}\}\subset\Sg^{n-2}(u,t_0)\times\{t_0\},\quad\text{for}~  t_0\in(-4,4).\]
For $ x_0\notin\Sg^{n-2}(u,t_0) $, let $ v $ be a tangent flow and $ \mu $ be a tangent measure of $ u $ at $ X_0=(x_0,t_0) $ with $ \mu=\f{1}{2}|\na v|^2\ud x\ud t+\nu $, where $ \nu $ is the defect measure. Moreover, we can assume $ \mu $ is spatially $ (n-1) $-symmetric. 

We claim: $ v $ and $ \mu $ are both spatially $ (n-1) $-symmetric. 

It follows from Proposition \ref{BlowUplimits}\ref{BlowUplimits6} that $ v $ and $ \mu $ are both backward self-similar, so we only need to show that $ v $ is backward invariant with respect to some $ (n-1) $-dimensional subspace. Assume that $ \mu $ is backward invariant with respect to $ V\in\bG(n,n-1) $. We intend to prove that $ v $ is also backward invariant with respect to $ V $. The proof follows almost the same arguments as in \cite[Lemma 2.18]{FWZ24}. Suppose that as $ r_i\to 0^+ $, 
\begin{align*}
v_i:=T_{X_0,r_i}u\wc v\text{ weakly in }H_{\loc}^1(\R^n\times\R,\cN)
\end{align*}
and 
\begin{align*}
\quad\f{1}{2}|\na v_i|^2\ud x\ud t\wc^*\mu\text{ in }\M(\R^n\times\R).
\end{align*}
Since $ \mu $ is backward self-similar and backward invariant with respect to $ V $, we have
\[
\Phi(\mu,(x,0),\rho)=\f{1}{2}\int_{T_{\rho}((x,0))}G_{(x,0)}(y,s)\ud\mu(y,s)
\]
is a constant function with respect to $ \rho>0 $ and $ x\in V $. Given Proposition \ref{BlowUplimits}\ref{BlowUplimits8}, especially formula \eqref{Phimu}, we see that $ v $ is backward self-similar at $ (x,0) $ with $ x\in V $. By  Lemma \ref{ConeSplittingCor}, we see that $ v $ is backward invariant with respect to $ V $. As a result, Proposition \ref{BlowUplimits}\ref{BlowUplimits3} and \ref{BlowUplimits4} imply that $ v $ is smooth for $ t\leq 0 $ and $ \nu\llcorner(\R^n\times\R_-)=0 $. Indeed, if $ v $ is not smooth for $ t\leq 0 $ or $ \nu\llcorner(\R^n\times\R_-)\neq 0 $, it will contradict the estimate of the local uniform $ \HH^{n-2} $-measure for the time slices of the concentration set since $ V $ is of dimension $ (n-1) $. Consequently, $ v $ is a constant map for $ t\leq 0 $. Moreover $ \mu\llcorner(\R^n\times\R_-)=0 $. It follows that 
\[
\lim_{i\to+\ift}E_-(u,X_0,r_i)=\lim_{i\to+\ift}E_-(v_i,0^{n,1},1)=0.
\]
Using Corollary \ref{PartialRegularityCor}, for sufficiently large $ i\in\Z_+ $, we see that $ u $ is smooth in the small neighborhood of $ X_0 $, placing $ X_0 $ in the regular set. This completes the proof.
\end{proof}

\section{Quantitative spatial stratification}\label{QuantitativeStratification}
In this section, we define the quantitative spatial stratification of harmonic map flows. The concept of quantitative stratification was first introduced by Cheeger and Naber in \cite{CN13a}. For some specific applications, we need to make certain adjustments to the definition.

We begin with the definition of quantitative spatial stratification.  

\begin{defn}[Quantitative spatial symmetry]
Suppose that $ r,R>0 $, $ X:=(x,t)\in P_R $ and $ u\in H_{\Lda}(P_R,\cN) $. We say that $ u $ is spatially $ (k,\va) $-symmetric in $ B_r(x) $ at the time $ t $ if $P_{r}(X)\subset P_{R}$ and there exists a spatially $ k $-symmetric Radon measure $ \mu\in\M(\R^n\times\R) $ satisfying
\[
\MM_{0^{n,1},1}\(\f{1}{2}|\na(T_{X,r}u)|^2\ud x\ud t,\mu\)<\va.
\]
\end{defn}

\begin{defn}[Quantitative spatial stratification]\label{quantitativeSpdef}
Let $ 0<r<R\leq 1 $ and $ u:P_4\to\cN $ be a suitable solution to the harmonic map flow \eqref{heatflow}. For each  $ k\in\Z\cap[0,n
] $, we define the spatial $ k $-th $ (\va,r) $-stratification of $ u $ at time $ t\in(-4,4) $ and scales $ r,R $ as
\[
\Sg_{\va;r,R}^k(u,t):=\{x\in B_2:\text{ for any }s\in[r,R],\,\,u\text{ is not spatially }(k+1,\va)\text{-symmetric in }B_s(x)\text{ at }t\}.
\]
Additionally, we define
\[
\Sg_{\va;0,R}^k(u,t):=\bigcap_{0<r<R}\Sg_{\va;r,R}^k(u,t).
\]
In other words,
\[
\Sg_{\va;0,R}^k(u,t)=\{x\in B_2:~\text{for any}~r\in(0,R], u\text{ is not spatially }(k+1,\va)\text{-symmetric in }B_r(x)\text{ at }t\}.
\]
For simplicity, we have the convention that $ \Sg_{\va;r}^k(u,t):=\Sg_{\va;r,1}^k(u,t) $
for any $ r\geq 0 $.
\end{defn}

\begin{rem}\label{inclusionSvak}
For $ \va,\va'>0 $, $ k,k'\in\Z\cap[0,n-1] $, $ 0<r<R\leq 1 $, and $ 0<r'<R'\leq 1 $. If $ \va\geq\va' $, $ k\leq k' $, $ r\leq r' $, and $ R\geq R' $, then $
\Sg_{\va;r,R}^k(u,t)\subset \Sg_{\va';r',R'}^{k'}(u,t) $ and $ \Sg_{\va;0,R}^k(u,t)\subset\Sg_{\va';0,R'}^{k'}(u,t) $.
\end{rem}

\begin{rem}\label{decomSkuseSva}
Similar to \cite[Formula (1.6)]{CN13b}, it follows from compactness arguments that for any $ R\in(0,1) $,
\be
\Sg^k(u,t)\cap B_2=\bigcup_{\va>0}\Sg_{\va;0,R}^k(u,t)=\bigcup_{\va>0}\bigcap_{0<r<R}\Sg_{\va;r,R}^k(u,t).\label{SgkB2}
\ee
\end{rem}

\begin{rem}
In the original definition of quantitative stratification introduced by Cheeger-Naber \cite{CN13a}, the authors did not restrict the scale $ R $. Indeed, they set $ R=1 $. However, we need to set precise upper scales. 
\end{rem}
We now present one useful property of the quantitative spatial symmetry.

\begin{lem}\label{BigBallSmallBall}
Let $ \Lda>0 $, $ k\in\Z\cap[0,n] $, $ r>0 $, $ X_0=(x_0,t_0)\in\R^n\times\R $ and $ u\in H_{\Lda}(P_{2r}(X_0),\cN) $. For any $ \va>0 $, there exists $ \delta\in(0,\va) $, depending only on $ \va,\Lda,n $, and $ \cN $ such that if $ u $ is spatially $ (k,\delta) $-symmetric in $ B_r(x_0) $ at $ t_0 $, then it is also spatially $ (k,\va) $-symmetric in $ B_{\f{r}{2}}(x_0) $ at $ t_0 $.
\end{lem}
\begin{proof}
Up to a scaling and a translation, we may assume $ X_0=0^{n,1} $ and $r=1$. Suppose that the result is not true, there exist $ \va>0 $, $ \delta_i\to 0 $, $ \{u_i\}\subset H_{\Lda}(P_2,\cN) $ such that $ u_i $ is spatially $ (k,\delta_i) $-symmetric in $ B_1 $ at $ t=0 $, but is not spatially $ (k,\va) $-symmetric in $ B_{\f{1}{2}} $ at $ t=0 $. Since $ u_i $ is spatially $ (k,\delta_i) $-symmetric in $ B_1 $ at $ t=0 $, there exist  spatially $ k $-symmetric Radon measures $ \{\mu_i\}\subset\M(\R^n\times\R) $ such that $ \MM_{0^{n,1},1}\(\f{1}{2}|\na u_i|^2\ud x\ud t,\mu_i\)<C(n)\delta_i $. Given Proposition \ref{Compactness},  we can assume that up to a subsequence
\be
\f{1}{2}|\na u_i|^2\ud x\ud t\wc^*\mu\quad\text{in }\M(P_1)\label{Converge12P1}
\ee
and
\be
\mu_i\wc^*\mu\quad\text{in }\M(P_1).\label{Converge12P2}
\ee
As in the proof of Proposition \ref{BlowUplimits}\ref{BlowUplimits7}, $ \mu $ must be spatially $ k $-symmetric in $ B_1 $ at $ t=0 $. This implies that for $ i\in\Z_+ $ to be sufficiently large, $ u_i $ must be spatially $ (k,\va) $-symmetric in $ B_{\f{1}{2}} $ at $ t=0 $, which is a contradiction to the assumption on $ u_i $.
\end{proof}
In the same spirit as in Lemma \ref{BigBallSmallBall}, we can prove the following result.
\begin{lem}\label{pinchLem}
Assume $ \Lda>0, X_0=(x_0,t_0)\in P_2 $ and $ u\in H_{\Lda}(P_4,\cN) $. For any $ \va>0 $, there exists $ \ga\in(0,1) $ depending only on $ \va,\Lda,n $, and $ \cN $ such that if $ r\in(0,\ga) $, and $ \cW(u,X_0,r,\ga r)<\ga $, then $ u $ is spatially $ (0,\va) $-symmetric in $ B_r(x_0) $ at $ t_0 $.
\end{lem}
\begin{proof}
Without loss of generality, assume that $ X_0=0^{n,1} $ after translation. Suppose not, then there exist $ \va>0 $, $ \ga_i\to 0^+ $, $ r_i\in(0,\ga_i) $, and a sequence $ \{u_i\}\subset H_{\Lda}(P_4,\cN) $ such that $ \cW(u_i,0^{n,1},r_i,\ga_ir_i)<\ga_i $, yet $ u_i $ is not spatially $ (0,\va) $-symmetric in $ B_{r_i} $ at $ t=0 $. 

Define $ v_i=T_{0^{n,1},r_i}u_i $. By Proposition \ref{BlowUplimits}\ref{BlowUplimits6}, there exists a measure $ \mu\in\M(\R^n\times\R) $, backward self-similar in $ \R^n\times(-4,0) $, such that $ \f{1}{2}|\na v_i|^2\ud x\ud t\wc^*\mu $ in $ \M(\R^n\times\R) $.  This implies $ \MM_{0^{n,1},1}(\f{1}{2}|\na v_i|^2\ud x\ud t,\mu)<\f{\va}{2} $for sufficiently large $ i\in\Z_+ $. Hence $ u_i $ is spatially $ (0,\f{\va}{2}) $-symmetric in $ B_{r_i} $ at $ t=0 $. It contradicts the assumption, proving the lemma.
\end{proof}

The following lemma deals with the inclusion property of strata for different scales $ 0\leq r<R\leq 1 $. 
\begin{lem}\label{InculsionDifferentScales}
Let $ \Lda>0 $, $ k\in\Z\cap[0,n] $, $ 0<R\leq\f{1}{2} $, $ t\in(-4,4) $, and $ u\in H_{\Lda}(P_4,\cN) $. For any $ \va>0 $, there exists $ \delta>0 $ depending only on $ \va,\Lda,n $, and $ \cN $ such that for any $ r\in[0,\f{R}{2}) $,
\be
\Sg_{\va;r,R}^k(u,t)\subset\Sg_{\delta;r,2R}^k(u,t)\subset\Sg_{\delta;r,R}^k(u,t).\label{deltarRkut}
\ee
\end{lem}
\begin{proof}
The second inclusion of \eqref{deltarRkut} follows directly from Remark \ref{inclusionSvak}. It remains to show the first one. For $ x\notin\Sg_{\delta,r,2R}^k(u,t) $, there is $ s\in[r,2R] $ such that $ u $ is spatially $ (k+1,\delta) $-symmetric in $ B_s(x) $ at $ t $. Choose $ \delta=\delta(\va,\Lda,n,\cN)\in(0,\va) $ as in Lemma \ref{BigBallSmallBall}. If $ s\in[r,2r] $, since $ \delta<\va $, we have $ x\notin\Sg_{\va;r,R}^k(u,t) $. On the other hand, for the case that $ s\in[2r,2R] $, it follows from Lemma \ref{BigBallSmallBall} that $ u $ is spatially $ (k+1,\va) $-symmetric in $ B_{\f{s}{2}}(x) $ at $ t $, showing that $ x\notin\Sg_{\va;r,R}^k(u,t) $ since $ \f{s}{2}\in[r,R] $. We complete the proof.
\end{proof}

We conclude this section with the following two lemmas; the first is a quantitative version of Proposition \ref{SgnSgn2}. 

\begin{lem}\label{FinaLemma}
Let $ \Lda>0, X_0=(x_0,t_0)\in P_2 $ and $ u\in H_{\Lda}(P_4,\cN) $. There exists $ \va>0 $, depending only on $ \Lda,n $, and $ \cN $ such that for $ r\in(0,\va) $, if $ u $ is spatially $ (n-1,\va) $-symmetric in $ B_r(x_0) $ at $ t_0 $, then $ r_u(X_0)\geq\va r $, where $ r_u(\cdot) $ denotes the regular scale given in Definition \ref{RegularScale}. In particular,
\be
\Sg_{\va;0,\va}^{n-2}(u,t_0)\cap (B_1\times\{t_0\})=\sing(u)\cap(B_1\times\{t_0\}).\label{FinallLemmaproperty}
\ee
\end{lem}
\begin{proof}
Without loss of generality, we let $ X_0=0^{n,1} $. Suppose that the claim fails. Then there exist a sequence $ \{u_i\}\subset H_{\Lda}(P_4,\cN) $, $ r_i\to 0^+ $, and $ \va_i\to 0^+ $ such that $ u_i $ is spatially $ (n-1,\va_i) $-symmetric in $ B_{r_i} $ at $ t=0 $, and $ r_{u_i}(0^{n,1})<\va_ir_i $. Let $ v_i=T_{0^{n,1},r_i}u_i $. Up to a subsequence, we have that $
\f{1}{2}|\na v_i|^2\ud x\ud t\wc^*\mu $ in $ \M(\R^n\times\R) $. Since each $ u_i $ is spatially $ (n-1,\va_i) $-symmetric in $ B_{r_i} $ at $ t=0 $, there are spatially $ (n-1) $-symmetric Radon measures $ \{\mu_i\}\subset\M(\R^n\times\R) $, satisfying
\[
\MM_{0^{n,1},1}\(\f{1}{2}|\na(T_{0^{n,1},r_i}u_i)|^2\ud x\ud t,\mu_i\)<\va_i.
\]
After taking a limit, $ \mu $ must be spatially $ (n-1) $-symmetric. As a result, we deduce from Proposition \ref{SgnSgn2} that $ \mu\llcorner(\R^n\times\R_-)=0 $. Applying Corollary \ref{PartialRegularityCor} and noting that $ r_i\to 0^+ $, there exists a constant $ c_0=c_0(\Lda,n,\cN)>0 $ such that for $ i\in\Z_+ $ being sufficiently large, we have that $ r_{u_i}(0^{n,1})\geq c_0r_i $. It contradicts the assumption $ r_{u_i}(0^{n,1})\leq \va_ir_i $ with $ \va_i\to 0^+ $. 

To show the equality \eqref{FinallLemmaproperty}, we first note that the right-hand side contains the left-hand side, due to \eqref{SgkB2}. On the other hand, if $ x\notin\Sg_{\va;0,\va}^{n-2}(u,0)\cap B_1 $, then there exists $ s\in(0,\va] $ such that $ u $ is spatially $ (n-1,\va) $-symmetric in $ B_s(x) $ at $ t=0 $. As a result, $ r_u((x,0))\geq \va s $, implying that $ (x,0)\notin\sing(u) $. Then we obtain \eqref{FinallLemmaproperty}.
\end{proof}
\begin{rem}
The property \eqref{FinallLemmaproperty} is the main reason we introduce the upper scale of quantitative spatial stratifications, as it is not clear whether $ \Sg_{\va;0,\va}^{n-2}(u,t_0) $ can be replaced by $ \Sg_{\va;0}^{n-2}(u,t_0) $.
\end{rem}
\begin{rem}
In the process of proving Lemma \ref{FinaLemma},  we have used the vanishing of spatially $ (n-1) $-symmetric tangent measures. 
\end{rem}
Finally, we provide an estimate of the regular scale under additional assumptions on the target manifold.

\begin{lem}\label{FinaLemma2}
Let $ \Lda>0, X_0=(x_0,t_0)\in P_2 $ and $ u\in H_{\Lda}(P_4,\cN) $.  If the target manifold $ \cN $ does not admit harmonic $ 2 $-spheres and quasi-harmonic $ 2 $-spheres, then there exist $ \va>0 $, depending only on $ \Lda,n $, and $ \cN $ such that if $ r\in(0,\va) $, and $ u $ is spatially $ (n-2,\va) $-symmetric in $ B_r(x_0) $ at $ t_0 $, then $ r_u(X_0)\geq\va r $.
\end{lem}
\begin{proof}
We choose $ X_0=0^{n,1} $, $ \va_i,r_i,u_i $, and $ v_i $ the same as in the proof of Lemma \ref{FinaLemma}. Since $ \cN $ does not admit harmonic $ 2 $-spheres, it follows from Proposition \ref{BlowUplimits}\ref{BlowUplimits5} that 
\[
v_i\to v\text{ strongly in }H_{\loc}^1(\R^n\times\R,\cN)\]
and
\[\f{1}{2}|\na v_i|^2\ud x\ud t \wc^*\mu:=\f{1}{2}|\na v|^2\ud x\ud t\text{ in } \M(\R^n\times\R).
\]
Recalling that $ u_i $ is spatially $ (n-2,\va_i) $-symmetric in $ B_{r_i} $ at $ t=0 $, we see that $ v $ is spatially $ (n-2) $-symmetric. Since $ \cN $ does not admit any quasi-harmonic $ 2 $-spheres, we obtain $ \mu\llcorner(\R^n\times\R_-)=0 $. By repeating the arguments in the proof of Lemma \ref{FinaLemma}, we get a contradiction.
\end{proof}

\section{Main covering lemma} \label{CoveringLemmas}
In this section, we provide several covering results related to quantitative spatial stratification, which serve as essential tools for proving the main theorems.

\begin{prop}\label{MainCoveringCor}
Let $ \va,\Lda>0 $, $ k\in\Z\cap[0,n-1] $, $ R\in(0,\f{1}{10}) $, and $ X_0=(x_0,t_0)\in P_2 $. Assume that $ u\in H_{\Lda}(P_4,\cN) $. Then, there exists $ \eta>0 $, depending only on $ \va,\Lda,n $, and $ \cN $ such that when $ r\in(0,\eta^2) $, the following properties hold.
\begin{enumerate}[label=$(\theenumi)$]
\item There is a collection of balls $ \{B_{Rr}(y)\}_{y\in\cC} $ such that
\be
\Sg_{\va;\eta Rr}^k(u,t_0)\cap B_r(x_0)\subset\bigcup_{y\in\cC}B_{Rr}(y).\label{Sgvaetasubset}
\ee
\item We have the estimate $ (\#\cC)R^k\leq C $, where $ C>0 $ depends only on $ \va,\Lda,n $, and $ \cN $.
\end{enumerate}
\end{prop}

Proposition \ref{MainCoveringCor} will be a consequence of the following covering result.

\begin{lem}[Main covering]\label{MainCovering}
Let $ \va,\Lda>0 $, $ k\in\Z\cap[0,n-1] $, $ R\in(0,1) $, and $ X_0=(x_0,t_0)\in P_2 $. There exists $ \eta\in(0,\f{1}{10}) $, depending only on $ \va,\Lda,n $, and $ \cN $ such that if $ u\in H_{\Lda}(P_4,\cN) $ and $ r\in(0,\eta^2) $, then the following properties hold. 

There is a collection of balls $ \{B_{r_y}(y)\}_{y\in\cC} $ with $ \cC\subset\Sg_{\va;\eta Rr}^k(u,t_0)\cap B_r(x_0) $, $ r_y\in[Rr,r) $ for any $ y\in\cC $, and 
\[
\Sg_{\va;\eta Rr}^k(u,t_0)\cap B_r(x_0)\subset\bigcup_{y\in\cC}B_{r_y}(y).
\]
Moreover, the following properties are satisfied.
\begin{enumerate}[label=$(\theenumi)$]
\item There exists $ C_{\op{M}}>0 $, depending only on $ \va,\Lda,n $, and $ \cN $ such that $ \sum_{y\in\cC}r_y^k\leq C_{\op{M}}r^k $.
\item For any $ y\in\cC $, either $ r_y=Rr $ or
\be
\sup_{z\in B_{2r_y}(y)}\Phi(u,(z,t_0),2r_y)\leq E-\f{\eta}{3}\quad\text{with}\quad E=\sup_{z\in B_{2r}(x_0)}\Phi(u,(z,t_0),2r).\label{Edefinition}
\ee
\end{enumerate}
\end{lem}
The proof of Lemma \ref{MainCovering} will be postponed to Section \ref{ProofofCoverSection}. In the rest of this section, we give the proof of Proposition \ref{MainCoveringCor} assuming Lemma \ref{MainCovering}.
\begin{proof}[Proof of Proposition \ref{MainCoveringCor} assuming Lemma \ref{MainCovering}]
For a sufficiently small $ \eta=\eta(\va,\Lda,n,\cN)>0 $ and $ r\in(0,\eta^2) $, we construct a collection of balls $ \{B_{r_y}(y)\}_{y\in\cC_i}:=\{B_{r_y}(y)\}_{y\in\cC_i^{(1)}\cup \cC_i^{(2)}} $ inductively, ensuring that
\be
\Sg_{\va;\eta Rr}^k(u,t_0)\cap B_r(x_0)\subset\bigcup_{y\in\cC_i^{(1)}}B_{r_y}(y)\cup\bigcup_{y\in\cC_i^{(2)}}B_{r_y}(y).\label{CoveringInductive}
\ee
The collection of balls satisfies the following properties.
\begin{enumerate}[label=$(\op{M}\theenumi)$]
\item\label{M1} $ \cC_i\subset\Sg_{\va;\eta Rr}^k(u,t_0)\cap B_r(x_0) $.
\item\label{M2} If $ y\in\cC_i^{(1)} $, then $ r_y=Rr $.
\item\label{M3} If $ y\in\cC_i^{(2)} $, then $ r_y>Rr $ and
\be
\sup_{z\in B_{2r_y}(y)}\Phi(u,(z,t_0),2r_y)\leq E-\f{i\eta}{3},\label{EnergyDrop}
\ee
where $ E $ is given by \eqref{Edefinition}.
\item\label{M4} For $ C_{\op{M}}>0 $, the constant from Lemma \ref{MainCovering}, we have
\be
\sum_{y\in\cC_i}r_y^k\leq (1+C_{\op{M}})^ir^k.\label{EstimateInductive}
\ee
\end{enumerate}

 By Lemma \ref{ComparisonEandPhi} and Lemma \ref{UniformBound}, $ E\leq C(\va,\Lda,n,\cN) $. The energy drop condition \eqref{EnergyDrop} ensures the process terminates after finitely many steps. Specifically, there exists $ i_0=i_0(\va,\Lda,n,\cN)\in\Z_+ $ such that $ E-\f{i\eta}{3}<0 $ for $ i\in\Z_{\geq i_0} $. Now we check that the collection $ \{B_{r_y}(y)\}_{y\in\cC_{i_0}} $ satisfies the desired properties.
 \begin{itemize}
 \item The inclusion in \eqref{Sgvaetasubset} follows from the inductive covering \eqref{CoveringInductive}.
 \item The estimate $ (\#\cC)R^k\leq C(\va,\Lda,n,\cN) $  follows from the bound \eqref{EstimateInductive} at the step $ i=i_0 $. 
\end{itemize}
Hence $ \{B_{r_y}(y)\}_{y\in\cC_{i_0}} $ is a covering satisfying Proposition \ref{MainCoveringCor}.
 
It remains to complete the inductive arguments. For $ i=1 $, properties \ref{M1}-\ref{M4} follow directly from applying Lemma \ref{MainCovering} to $ B_r(x_0) $. Assuming that \ref{M1}-\ref{M4} hold for some $ i\in\Z_+ $, we construct the next step $ i+1 $. For any $ \zeta\in\cC_i^{(2)} $, we have $ r_{\zeta}>Rr $. Applying Lemma \ref{MainCovering} to $B_{r_{\zeta}}(\zeta)$ with $R$ replaced by $R' = \frac{Rr}{r_{\zeta}} \in (0,1)$ provides a collection $\{B_{r_y}(y)\}_{y \in \cC_{\zeta,i}}$ with $\cC_{\zeta,i} \subset \Sg_{\va;\eta Rr}^k(u,t_0) \cap B_{r_{\zeta}}(\zeta)$ such that
\[
\Sg_{\va;\eta Rr}^k(u,t_0)\cap B_{r_{\zeta}}(\zeta)\subset\bigcup_{y\in\cC_{\zeta,i}}B_{r_y}(y)=\bigcup_{y\in\cC_{\zeta,i}^{(1)}}B_{r_y}(y)\cup\bigcup_{y\in\cC_{\zeta,i}^{(2)}}B_{r_y}(y).
\]
Moreover, if $ y\in\cC_{\zeta,i}^{(1)} $, then $ r_y=Rr $, and if $ y\in\cC_{\zeta,i}^{(2)} $, then
\[
\sup_{z\in B_{2r_y}(y)}\Phi(u,(z,t_0),2r_y)\leq\sup_{z\in B_{2r_{\zeta}}(\zeta)}\Phi(u,(z,t_0),2r_{\zeta})-\f{\eta}{3}\leq E-\f{(i+1)\eta}{3}.
\]
Also, we have the estimate
\be
\sum_{y\in\cC_{\zeta,i}}r_y^k\leq C_{\op{M}}(\va,\Lda,n,\cN)r_{\zeta}^k.\label{CxiCx2}
\ee
Define $ \{\cC_{i+1}^{(j)}\}_{j=1,2} $ as
\[
\cC_{i+1}^{(1)}=\cC_i^{(1)}\cup\bigcup_{y\in \cC_i^{(2)}}\cC_{y,i}^{(1)},\quad\cC_{i+1}^{(2)}=\bigcup_{y\in \cC_i^{(2)}}\cC_{y,i}^{(2)},\quad\text{and}\quad\cC_{i+1}=\cC_{i+1}^{(1)}\cup\cC_{i+1}^{(2)}.
\]
According to \eqref{CxiCx2} and \ref{M4} for $ i $, 
\begin{align*}
\sum_{y\in\cC_{i+1}}r_y^k&\leq\sum_{y\in\cC_i^{(1)}}r_y^k+\sum_{\zeta\in\cC_i^{(2)}}\sum_{y\in\cC_{\zeta,i}}r_y^k\leq (1+C_{\op{M}})\(\sum_{y\in\cC_i}r_y^k\)\leq (1+C_{\op{M}})^{i+1}r^k,
\end{align*}
which completes the proof.
\end{proof}

\section{Reifenberg-type theorems} \label{Reifenbergtype}

In this section, we present the first ingredient in the proof of Lemma \ref{MainCovering}, which is the Reifenberg-type theorems proved by Naber-Valtorta in \cite{NV17}.

\begin{defn}\label{displacementk}
Let $ k\in\Z\cap[0,n-1] $, $ 0<r\leq 1 $, and $ U\subset\R^n $ be a bounded open set. Assume that $ \mu $ is a finite Radon measure on $ U $, that is, $ \mu(U)<+\ift $. For $ x_0\in U $ and $ 0<r<\dist(x_0,\pa U) $, we define the $ k $-dimensional displacement as
\[
D_{\mu}^k(x_0,r):=\min_{L\in\bA(n,k)}\(r^{-k-2}\int_{B_r(x_0)}\dist^2(y,L)\ud\mu(y)\).
\]
\end{defn}

The following theorem provides Reifenberg-type estimates for a finite sum of Dirac measures with distinct weights.

\begin{thm}[\cite{NV17}, Theorem 3.4]\label{Rei1}
Let $ k\in\Z\cap[0,n-1] $, $ 0<r\leq 1 $, and $ x_0\in\R^n $. Assume that $ \{B_{r_y}(y)\}_{y\in\cD}\subset B_{2r}(x_0) $ is a collection of pairwise disjoint balls with $ \cD\subset B_r(x_0) $. Define 
\[
\mu:=\sum_{y\in\cD}\w_kr_y^k\delta_y,
\]
where $ \w_k $ denotes the volume of a $ k $-dimensional unit ball. There exist $ \delta_{\op{R}},C_{\op{R}}>0 $, depending only on $ n $ such that if
\be
\int_{B_t(x)}\(\int_0^tD_{\mu}^k(y,s)\f{\ud s}{s}\)\ud\mu(y)<\delta_{\op{R}}t^k\label{Reicon}
\ee
for any $ B_t(x)\subset B_{2r}(x_0) $ with $ t>0 $, then $
\mu(B_r(x_0))\leq C_{\op{R}}r^k $.
\end{thm}

This theorem also has a discrete version, described in the following.

\begin{cor}[\cite{NV17}, Remark 3.11]\label{ReifenbergRem}
The result is still true if we replace the condition \eqref{Reicon} by
\[
\sum_{r_i\leq 2r} \int_{B_r(x)}D_{\mu}^k(y, r_i)\ud\mu(y)<\delta_{\op{R}}'r^k,
\]
where $ r_i=2^{-i}r $ and $ \delta_{\mathrm{R}}'>0 $ depends only on $ n $.
\end{cor}

The following theorem provides a criterion for the rectifiability of sets.

\begin{thm}[\cite{AT15}, Corollary 1.3]\label{Rei2}
Let $ S\subset\R^n $ be a $ \HH^k $-measurable set. $ S $ is rectifiable if and only if for $ \HH^k $-a.e. $ x\in S $,
\[
\int_0^1D_{\HH^k\llcorner S}^k(x,s)\f{\ud s}{s}<+\ift.
\]
\end{thm}

\section{\texorpdfstring{$ L^2 $}{}-best approximation estimates}\label{L2best}

To apply the theorems in Section \ref{Reifenbergtype}, we need to establish the connection between the difference of densities $ \Phi $ at two scales, as defined by \eqref{Difference}, and the displacements outlined in Definition \ref{displacementk}. The proposition highlighted in the introduction represents a key innovation in our paper. It is the first to connect Reifenberg-type results with harmonic map flows.

\begin{prop}\label{L2BestProp}
Let $ \va,\Lda>0 $, $ k\in\Z\cap[0,n-1] $, $ r\in(0,\f{1}{4}) $, and $ X_0=(x_0,t_0)\in P_2 $. Assume that $ u\in H_{\Lda}(P_4,\cN) $. There exists $ \delta>0 $, depending only on $ \va,\Lda,n $, and $ \cN $ such that if for some $ r\in(0,\delta) $, $ u $ is spatially $ (0,\delta) $-symmetric in $ B_{2r}(x_0) $ at $ t_0 $ and is not spatially $ (k+1,\va) $-symmetric in $ B_{2r}(x_0) $ at $ t_0 $, then for any $ \mu\in\M(B_r(x_0)) $ with $ \mu(B_r(x_0))<+\ift $, it holds that
\[
D_{\mu}^k(x_0,r)\leq\f{C}{r^k}\int_{B_r(x_0)}\left[\cW\(u,(y,t_0),2r,\f{r}{2}\)+r\right]\ud\mu(y),
\]
where $ C>0 $ depends only on $ \va,\Lda,n $, and $ \cN $.
\end{prop}

\subsection{Preparations for the proof of Proposition \ref{L2BestProp}}

For a point $ x_0 \in \mathbb{R}^n $ and a finite Radon measure $ \mu $ on $ B_r(x_0) $, we define the center of mass of $ \mu $ as
\be
x_{\op{cm}}:=x_{\op{cm}}(\mu)=\dashint_{B_r(x_0)}y\ud\mu(y).\label{centermass}
\ee

\begin{defn}\label{defvj}
We inductively define the sequence $ \{(\lda_i,v_i)\}_{i=1}^n\subset\R_{\geq 0}\times\mathbb{S}^{n-1} $ as follows. Set
\[
\lda_1:=\lda_1(\mu):=\max_{|v|^2=1}\dashint_{B_r(x_0)}|(y-x_{\op{cm}})\cdot v|^2\ud\mu(y),
\]
and let $ v_1:=v_1(\mu) $ be a unit vector satisfying:
\be
\lda_1=\dashint_{B_r(x_0)}|(y-x_{\op{cm}})\cdot v_1|^2\ud\mu(y).\label{ldaiform1}
\ee
Given $ \{(\lda_j,v_j)\}_{j=1}^i $, we define $ (\lda_{i+1},v_{i+1})\in\R_{\geq 0}\times\mathbb{S}^{n-1} $ so that
\be
\lda_{i+1}:=\max_{\substack{|v|^2=1,\,\,v\cdot v_j=0,\\j\in\Z\cap[1,i]}}\dashint_{B_r(x_0)}|(y-x_{\op{cm}})\cdot v|^2\ud\mu(y)=\dashint_{B_r(x_0)}|(y-x_{\op{cm}})\cdot v_{i+1}|^2\ud\mu(y).\label{ldaiform2}
\ee
\end{defn}
According to standard variational theories, $ \{v_i\}_{i=1}^n $ is an orthonormal basis of $ \R^n $, and 
\be
\lda_1\geq\lda_2\geq...\geq\lda_n\geq 0.\label{ldaorder}
\ee
Using this sequence $ \{(\lda_i,v_i)\}_{i=1}^n $, we can express $ D_{\mu}^k(x_0,r) $ through the following lemma.

\begin{lem}[\cite{NV17}, Lemma 7.4]\label{NVlem7.4} 
Let $ \{(\lda_i,v_i)\}_{i=1}^n $ be given in Definition \ref{defvj}. If $ \mu $ is a Radon measure on $ B_r(x_0) $ with $ \mu(B_r(x_0))<+\ift $, then for any $ k\in\Z\cap[1,n] $, 
\be
\min_{L\in\bA(n,k)}\dashint_{B_r(x_0)}\dist^2(y,L)\ud\mu(y)=\dashint_{B_r(x_0)}\dist^2(y,L_k)\ud\mu(y)=\sum_{i=k+1}^n\lda_i,\label{mnkrepre}
\ee
where $ L_k:=x_{\op{cm}}+\op{span}\{v_i\}_{i=1}^k\in\bA(n,k) $, and $ \sum_{i=n+1}^n\lda_i=0 $ as a convention.
\end{lem}

From Definition \ref{displacementk} and the above results, it follows that
\be
D_{\mu}^k(x_0,r)=\f{\mu(B_r(x_0))}{r^{k+2}}\(\sum_{i=k+1}^n\lda_i\)\leq \f{(n-k)\lda_{k+1}}{r^{k+2}}\mu(B_r(x_0)).\label{DmuxrRepresentation}
\ee

\begin{lem}[\cite{NV17}, Lemma 7.5]\label{NVlem7.5}
For $ \{(\lda_i,v_i)\}_{i=1}^n $ as in Definition \ref{defvj} and a Radon measure $ \mu $ on $ B_r(x_0) $ with $ \mu(B_r(x_0))< +\ift $, we have
\be
\dashint_{B_r(x_0)}((y-x_{\op{cm}})\cdot v_i)(y-x_{\op{cm}})\ud\mu(y)=\lda_iv_i\label{El}
\ee
for any $ i\in\Z\cap[1,n] $.
\end{lem}

\subsection{Proof of Proposition \ref{L2BestProp}}

In this subsection, we will give the proof of  Proposition \ref{L2BestProp}. First, we use Lemma \ref{NVlem7.5} and Proposition \ref{Monotonicity} to derive a crucial result to prove the $ L^2 $-best approximation estimate.

\begin{lem}\label{VnablauLeqTwoPhi}
Let $ \Lda>0 $, $ r\in(0,\f{1}{4}) $ and $ X_0=(x_0,t_0)\in P_2 $. Assume that $ u\in H_{\Lda}(P_4,\cN) $ and $ \{(\lda_i,v_i)\}_{i=1}^n $ are given by \eqref{defvj}. Then, there exists $ C>0 $, depending only on $ \Lda $ and $ n $ such that for any $ i\in\Z\cap[1,n] $,
\begin{align*}
&\lda_i\(\frac{1}{r^{n}}\int_{t_0-4r^2}^{t_0-2r^2}\int_{B_r(x_0)}|v_i\cdot\na u|^2\ud z\ud s\)\leq\f{Cr^2}{\mu(B_r(x_0))}\int_{B_r(x_0)}\left[\cW\(u,(y,t_0),2r,\f{r}{2}\)+r\right]\ud\mu(y).
\end{align*}
\end{lem}
\begin{proof}
After translation, we may assume that $ X_0=0^{n,1} $. Multiplying both sides of \eqref{El} by $ \na u(z,s)\exp(\f{|z|^2}{8s}) $ with $ (z,s) \in B_r\times[-4r^2,-r^2] $ yields
\be
\lda_i(v_i\cdot\na u(z,s))\exp\(\f{|z|^2}{8s}\)=\dashint_{B_r}[(y-x_{\op{cm}})\cdot v_i][(y-x_{\op{cm}})\cdot\na u(z,s)]\exp\(\f{|z|^2}{8s}\)\ud\mu(y).\label{ldaivinablau}
\ee
Recalling the definition of $ x_{\op{cm}} $ given by \eqref{centermass}, we have
\[
\dashint_{B_r}[(y-x_{\op{cm}})\cdot v_i][(x_{\op{cm}}-z)\cdot\na u(z,s)-2s\pa_su(z,s)]\exp\(\f{|z|^2}{8s}\)\ud\mu(y)=0.
\]
This, together with \eqref{ldaivinablau} and Cauchy's inequality, implies that
\begin{align*}
&\lda_i^2|(v_i\cdot\na u(z,s))|^2\exp\(\f{|z|^2}{4s}\)\\
\leq&\(\dashint_{B_r}|(y-x_{\op{cm}})\cdot v_i|^2\ud\mu(y)\)\left[\dashint_{B_r}|(y-z)\cdot\na u(z,s)-2s\pa_su(z,s)|^2\exp\(\f{|z|^2}{4s}\)\ud\mu(y)\right].
\end{align*}
Combining \eqref{ldaiform1}, \eqref{ldaiform2}, and the property that
\[
\exp\(-\f{1}{4}\)\leq\exp\(\f{|z|^2}{4s}\)\leq \exp\(\f{5}{4}\)\exp\(\f{|z-y|^2}{4s}\)
\]
for any $ (z,s)\in B_r\times[-4r^2,-r^2] $, we have
\[
\lda_i|v_i\cdot\na u(z,s)|^2\leq \exp\(\f{5}{4}\)\dashint_{B_r}|(y-z)\cdot\na u(z,s)-2s\pa_su(z,s)|^2\exp\(\f{|y-z|^2}{4s}\)\ud\mu(y).
\]
Integrating both sides of the above inequality with respect to $ (z,s) $ in $ B_r\times[-4r^2,-r^2] $, we deduce that
\begin{align*}
&\lda_i\int_{-4r^2}^{-r^2}\int_{B_r}|v_i\cdot\na u(z,s)|^2\ud z\ud s\\
\leq & Cr^2\int_{-4r^2}^{-r^2}
\int_{B_r}\left[\dashint_{B_r}\f{|(z-y)\cdot\na u(z,s)+2s\pa_su(z,s)|^2}{|s|}\exp\(\f{|z-y|^2}{4s}\)\ud\mu(y)\right]\ud z\ud s.
\end{align*}
Then Proposition \ref{Monotonicity} completes the proof.
\end{proof}

The other lemma we need in the proof of Proposition \ref{L2BestProp} is as follows.

\begin{lem}\label{LemGeqdelta}
Let $ \Lda,\va>0 $, $ k\in\Z\cap[0,n-1] $, $ r\in(0,\f{1}{4}) $, and $ X_0=(x_0,t_0)\in P_2 $. Assume that $ u\in H_{\Lda}(P_4,\cN) $. There exist $ \delta>0, \delta'>0$ depending only on $ \va,\Lda,n $, and $ \cN $ such that if for some $ r\in(0,\delta) $, $ u $ is spatially $ (0,\delta) $-symmetric in $ B_{2r}(x_0) $ at $ t_0 $ and is not spatially $ (k+1,\va) $-symmetric in $ B_r(x_0) $ at $ t_0 $, then 
\[
\inf_{V\in\bG(n,k+1)}\(r^{-n}\int_{t_0-4r^2}^{t_0-r^2}\int_{B_r(x_0)}|V\cdot\na u|^2\ud x\ud t\)>\delta'.
\]
\end{lem}
\begin{proof}
Without loss of generality, we can assume $ X_0=0^{n,1} $ by applying a translation. Suppose that the statement does not hold. In that case, there exist $ \va>0 $, $ \delta_i\to 0^+ $, $ \delta_i'\to 0^+$, $ r_i\in(0,\delta_i) $, a sequence of suitable solutions $ \{u_i\}\subset H_{\Lda}(P_4,\cN) $, and $ \{V_i\} \subset\bG(n,k+1) $ such that each $ u_i $ is spatially $ (0,\delta_i) $-symmetric in $ B_{2r_i} $ at $ t=0 $, but not spatially $ (k+1,\va) $-symmetric in $ B_{r_i} $ at $ t=0 $. Moreover, for each $i$, $u_{i}$ satisfies
\[
r_i^{-n}\int_{-4r_i^2}^{-r_i^2}\int_{B_{r_i}}|V_i\cdot\na u_i|^2\ud x\ud t<\delta_i'.
\]
Given the spatially $ (0,\delta_i) $-symmetry of $ u_i $ in $ B_{2r_i} $ at $ t=0 $, there exists a sequence of Radon measures $ \{\mu_i\}\subset\M(\R^n\times\R) $ which are backward self-similar such that
\be
\MM_{0^{n,1},2}\(\f{1}{2}|\na v_i|^2\ud x\ud t,\mu_i\)<C(n)\delta_i,\label{d01nablavimui}
\ee
where $ v_i=T_{0^{n,1},r_i}u_i $. Up to a subsequence, we assume that $ V_i\to V\in\bG(n,k+1) $, and
\be
\begin{aligned}
\mu_i\wc^*\mu,\,\,&\quad\text{in }\M(\R^n\times\R),\\
\f{1}{2}|\na v_i|^2\ud x\ud t\wc^*\mu',&\quad\text{in }\M(\R^n\times\R).\label{Convergencemuwtmu}
\end{aligned}
\ee
Due to the reasoning in Proposition \ref{BlowUplimits}\ref{BlowUplimits6}, we see that $ \mu $ remains backward self-similar. From the above convergences, $ \mu' $ is also backward self-similar in the region $ B_2\times(-4, 0) $. Moreover, Proposition \ref{BlowUplimits}\ref{BlowUplimits7} shows that $ \mu' $ is backward invariant with respect to $ V $ in $ B_1\times(-4,-1) $. Using arguments similar to those in the proof of \cite[Lemma 7.2]{NV17}, we can demonstrate that this backward self-similarity of $ \mu' $ extends the invariance with respect to $ V $ in the domain $ B_2\times(-4,0) $. Specifically, $ \mu' $ is spatially $ (k+1) $-symmetric in $ P_1 $. For sufficiently large $ i\in\Z_+ $, the convergence \eqref{Convergencemuwtmu} implies that $ u_i $ must be spatially $ (k+1, \va) $-symmetric in $ B_{r_i} $ at $ t=0 $. However, this contradicts our initial assumption that $ u_i $ is not spatially $ (k+1,\va) $-symmetric in $ B_{r_i} $ at $ t=0 $. Thus, the original statement  holds.
\end{proof}

\begin{proof}[Proof of Proposition \ref{L2BestProp}]
In view of \eqref{ldaorder} and Lemma \ref{VnablauLeqTwoPhi}, we have
\begin{align*}
&\lda_{k+1}\(r^{-n}\int_{t_0-4r^2}^{t_0-2r^2}\int_{B_r(x_0)}|V_{k+1}\cdot\na u|^2\ud z\ud s\)\\
&\quad\quad\quad\quad\leq\f{Cr^2}{\mu(B_r(x_0))}\int_{B_r(x_0)}\left[\cW\(u,(y,t_0),2r,\f{r}{2}\)+r\right]\ud\mu(y),
\end{align*}
where $ V_{k+1}=\op{span}\{v_i\}_{i=1}^{k+1} $. The desired property follows from \eqref{DmuxrRepresentation} and Lemma \ref{LemGeqdelta}.
\end{proof}

\subsection{A key technical lemma}

At the end of this section, we use Reifenberg-type theorems and $ L^2 $-best approximation estimates to bound a specific collection of balls centered on the singular set. This result is crucial in the construction of the coverings.

\begin{lem}\label{ReifenbergLemma}
Assume $ \Lda,\va>0 $, $ R\in(0,\f{1}{10}) $, $ X_0=(x_0,t_0)\in P_2 $, and $ u\in H_{\Lda}(P_4,\cN) $. There exists $ \eta>0 $ depending only on $ \va,\Lda,n $, and $ \cN $ such that if $ r\in(0,\eta) $, then the following result holds. 

Suppose that 
\be
\sup_{y\in B_r(x_0)}\Phi(u,(y,t_0),2r)\leq E.\label{Brx0Phi}
\ee
Consider a collection of disjoint balls $ \{B_{2r_y}(y)\}_{y\in\cD} $, where $ \cD\subset\Sg_{\va;Rr}^k(u,t_0)\cap B_r(x_0) $, and for any $ y\in \cD $, the radius satisfies $ Rr\leq r_y\leq r $ and
\be
\Phi(u,(y,t_0),\eta r_y)\geq E-\eta,\label{Phiyt0geq}
\ee
then there is $ C>0 $, depending only on $ n $ such that
\[
\sum_{y\in\cD}r_y^k\leq Cr^k.
\]
\end{lem}
\begin{proof}
Without loss of generality, we can assume that $ X_0=0^{n,1}$. For each $ i\in\Z_+ $, we define $ r_i:=2^{-i}r $ and $
\mu_i:=\sum_{y\in\cD,\,\,r_y\leq r_i}\w_kr_y^k\delta_y $. Then the desired estimate is equivalent to $ \mu_0(B_r) \leq C(n)r^{k} $. By the definition of $\mu_{i}$, it is easy to check if $ y\in\supp\mu_i $ and $ j\in\Z_{\geq i} $, $ D_{\mu_i}^k(y,r_j) $ satisfies
\be
D_{\mu_i}^k(y,r_j)=\left\{\begin{aligned}
&D_{\mu_j}^k(y,r_j)&\text{ if }&y\in\supp\mu_j,\\
&0&\text{ if }&y\notin\supp\mu_j.
\end{aligned}\right.\label{Dji}
\ee

We claim: there exists $\eta$ such that for $ y\in\cD $ and $ r_i<\f{r}{4} $, it holds
\be
D_{\mu_i}^k(y,r_i)\leq\f{C(\va,\Lda,n,\cN)}{r_i^k}\int_{B_{r_i}(y)}\left[\cW\(u,(y,0),2r_i,\f{r_i}{2}
\)+r_i\right]\ud\mu_i(z).\label{DisplacementEstimates}
\ee

For $ y\in\cD $ and $ r_i<\f{r}{4} $, we get from conditions \eqref{Brx0Phi}, \eqref{Phiyt0geq} and Proposition \ref{Monotonicity} that $ \cW(u,(y,0),8r_i,\eta r_i)\leq\eta $. By choosing a sufficiently small $ \eta= \eta(\delta,\Lda,n,\cN)>0 $ for a given $ \delta>0 $, Lemma \ref{pinchLem} implies that if $ r\in(0,\eta) $, then $ u $ is spatially $ (0,\delta) $-symmetric in $ B_{2r_i}(y) $ at $ t=0 $. However, since $ \cD\subset \Sg_{\va;Rr}^k(u,0)\cap B_r $ and $ r_i \geq Rr $, the definition of quantitative spatial stratification ensures that $ u $ is not spatially $ (k+1,\va) $-symmetric in $ B_{r_i}(y) $ at $ t=0 $. Shrinking $ \eta=\eta(\va,\Lda,n,\cN)>0 $ so that Proposition \ref{L2BestProp} is applicable, we obtain the claimed bound.

To proceed,  we  use an inductive argument to establish that up to basic covering arguments, there exists a constant $ C_0=C_0(n)>0 $ such that 
\be
\mu_i(B_{r_i}(x))\leq C_0(n)r_i^k\quad\text{for any }x\in B_r~\text{and}~i\geq 2.\label{InductivelyProof}
\ee

 For large $ i\in\Z_+ $, where $ r_i<Rr $, the definition of $\mu_{i}$ implies $ \mu_i \equiv 0 $, making the inequality trivial. It serves as the base case for induction. Assume that \eqref{InductivelyProof} holds for all $ j\in\Z_{\geq i+1} $. We need to verify it for $ i $. First, we establish a preliminary estimate
\be
\mu_j(B_{4r_j}(x))\leq C(n)C_0(n)r_j^k\quad\text{for any }x\in B_r\text{ and }j\geq i-2,\label{RoughBound}
\ee
where $ C_0>0 $ is given in \eqref{InductivelyProof}. For $ x \in B_r $, we write
\[
\mu_j(B_{4r_j}(x))=\mu_{j+2}(B_{4r_j}(x))+\sum_{y\in\cD\cap B_{4r_j}(x),\,\, r_{j+2}<r_y\leq r_j}\w_kr_y^k.
\]
With the help of the disjointness of the balls in the collection $ \{B_{2r_y}(y)\}_{y\in\cD} $, \eqref{RoughBound} follows directly. 

We now fix $ x\in B_r $. Using \eqref{Dji}, \eqref{DisplacementEstimates}, together with Fubini's theorem, we compute
\begin{align*}
&\sum_{r_j\leq 2r_i}\int_{B_{r_i}(x)}D_{\mu_i}^k(z,r_j)\ud\mu_i(z)\\
=&\sum_{r_j\leq 2r_i}\int_{B_{r_i}(x)}D_{\mu_j}^k(z,r_j)\ud\mu_j(z)\\
\leq&\sum_{r_j\leq 2r_i}\f{C}{r_j^k}\int_{B_{r_i}(x)}\left\{\int_{B_{r_j}(z)}\left[\cW\(u,(y,0),2r_j,\f{r_j}{2}\)+r_j\right]\ud\mu_j(y)\right\}\ud\mu_j(z)\\
\leq &\sum_{r_j\leq 2r_i}\int_{B_{r_i+r_j}(x)}\f{C\mu_j(B_{r_j}(y))}{r_j^k}\left[\cW\(u,(y,0),2r_j,\f{r_j}{2}\)+r_j\right]\ud\mu_j(y),
\end{align*}
where for the first equality, we have used the property given in \cite[Lemma 6.3]{WWW25}. By \eqref{RoughBound}, it follows that
\begin{align*}
&\quad\sum_{r_j\leq 2r_i}\int_{B_{r_i}(x)}D_{\mu_i}^k(z,r_j)\ud\mu_i(z)\\
&\leq C\int_{B_{4r_i}(x)}\sum_{r_j\leq 2r_i}\left[\cW\(u,(y,t_0),2r_j,\f{r_j}{2}\)+r_j\right]\ud\mu_j(y)\\
&\leq C\(\sum_{y\in B_{4r_i}(x)\cap\supp\mu_i}r_y^k(\cW\(u,(y,0),8r_j\)+r_j)\)\\
&\leq C\eta\mu_i(B_{4r_i}(x))\\
&\leq C(\va,\Lda,n,\cN)\eta r_i^k.
\end{align*}
If we select a sufficiently small $ \eta=\eta(\va,\Lda,n,\cN)>0 $, then \eqref{InductivelyProof} follows from Corollary \ref{ReifenbergRem}.
\end{proof}

\section{Dichotomy property of the quantitative spatial stratification}\label{DichotomySection}

In the present section, we gather several properties related to the quantitative spatial stratification of harmonic map flows. Before giving the main result, we introduce a useful concept.
\begin{defn}[Effectively spanned subspace]
For $ k\in\Z\cap[1,n] $, consider points $ \{x_i\}_{i=0}^k \subset\R^n $ and $ \rho>0 $. We say that these points $ \rho $-effectively span the subspace $ L= x_0+\op{span}\{x_i-x_0\}_{i=1}^k\in\bA(n,k) $ if, $ |x_1-x_0|\geq 2\rho $ and for any $ i\in\Z\cap[2,k] $,
\[
\dist(x_i,x_0+\op{span}\{x_1-x_0,...,x_{i-1}-x_0\})\geq 2\rho.
\]
We also call such points $ \rho $-independent.
\end{defn}

The main result in this section is the following dichotomy result.
\begin{prop}[Dichotomy result]\label{Dichotomy}
Let $ \va,\eta',\rho,\ga\in(0,\f{1}{4}) $, $ E',\Lda>0 $, $ k\in\Z\cap[1,n-1] $, and $ X_0=(x_0,t_0)\in P_2 $. Assume that $ u\in H_{\Lda}(P_4,\cN) $. Then, there exists $ \eta_0\ll\rho $, depending only on $ E',\va,\eta',\ga,\Lda,n,\cN $, and $ \rho $ such that the following properties hold.

Suppose that
\[
\sup_{y\in B_r(x_0)}\Phi(u,(y,t_0),2r)\leq E\in[0,E'].
\]
If $ \eta,r\in(0,\eta_0) $, at least one of the following two possibilities occurs.
\begin{enumerate}[label=$(\theenumi)$]
\item For any $ y\in \Sg_{\va;\eta r}^k(u,t_0)\cap B_r(x_0) $, $ \Phi(u,(y,t_0),2\ga\rho r)\geq E-\eta' $.
\item There is $ L\in\bA(n,k-1) $ such that
\be
\{y\in B_r(x_0):\Phi(u,(y,t_0),2\eta r)\geq E-\eta\}\subset B_{\rho r}(L).\label{Choice2}
\ee
If $ k=0 $, we adopt the convention that the set on the left-hand side of \eqref{Choice2} is empty.
\end{enumerate}
\end{prop}
\begin{rem}\label{Conventionll}
The notation $ \eta_0\ll\rho $ means that $ \eta_0 $ can be chosen such that $ \eta_0<c_0\rho $ for any fixed $ c_0>0 $.
\end{rem}
First, we prove a preliminary lemma.
\begin{lem}\label{DichotomyLem}
Let $ \va,\eta',\rho,\ga\in(0,\f{1}{4}) $, $ E',\Lda>0 $, and $ X_0=(x_0,t_0)\in P_2 $. Assume that $ u\in H_{\Lda}(P_4,\cN) $. Then, there exist constants $ \eta_0,\beta\ll\rho $, depending only on $ E',\va,\eta',\ga,\Lda,n,\cN $, and $ \rho $ such that the following properties hold. 

Suppose that
\[
\sup_{y\in B_r(x_0)}\Phi(u,(y,t_0),2r)\leq E\in[0,E'].
\]
If $ \eta,r\in(0,\eta_0) $, and there are $ \rho r $-independent points $ \{y_i\}_{i=0}^k\subset B_r(x_0) $ such that 
\[
\Phi(u,(y_i,t_0),2\eta r)\geq E-\eta\quad\text{ for any }i\in\Z\cap[0,k],
\]
then for $ L=y_0+\op{span}\{y_i-y_0\}_{i=1}^k $, 
\be
\Phi(u,(y,t_0),2\ga\rho r)\geq E-\eta'\quad\text{ for any }y\in B_{\beta r}(L)\cap B_r(x_0),\label{EnergyGeq}
\ee
and
\be
\Sg_{\va;\eta r}^k(u,t_0)\cap B_r(x_0)\subset B_{\beta r}(L).\label{SubsetSg}
\ee
\end{lem}
\begin{proof}
Assume $ X_0=0^{n,1} $ after a translation. 

\emph{Step 1. Proof of the property \eqref{EnergyGeq}.} Suppose that the claim fails. Then, there exist sequences $ \{u_j\}\subset H_{\Lda}(P_4,\cN) $, $ \eta_j\to 0^+ $, $ \beta_j\to 0^+ $, points $ \{y_{i,j}\}_{i=0}^k\subset B_{r_j} $ that are $ \rho r_j $-independent and $ \{r_j\}\subset(0,\eta_j) $ satisfying
\begin{gather}
\sup_{y\in B_{r_j}}\Phi(u_j,(y,0),2r_j)\leq E_j\in[0,E'],\label{Phiuj1}\\
\Phi(u_j,(y_{i,j},0),2\eta_jr_j)\geq E_j-\eta_j\quad\text{ for any }i\in\Z\cap[0,k],\label{Phiujyj}
\end{gather}
but for some $ x_j\in B_{\beta_jr_j}(L_j)\cap B_{r_j} $, 
\be
\Phi(u_j,(x_j,0),2\ga\rho r_j)<E_j-\eta',\quad\eta'>0,\label{PhiujEjeta}
\ee
where $ L_j=y_{0,j}+\op{span}\{y_{i,j}-y_{0,j}\}_{i=1}^k $. Define
\[
\wt{L}_j=\f{y_{0,j}}{r_j}+\op{span}\left\{\f{y_{i,j}-y_{0,j}}{r_j}\right\}_{i=1}^k.
\]

As a result, there exist $ v\in H_{\loc}^1(\R^n\times\R,\cN) $ and $ \nu\in\M(\R^n\times\R) $ such that, up to a subsequence,
\be
\begin{gathered}
v_j:=T_{0^{n,1},r_j}u_j\wc v\text{ weakly in }H_{\loc}^1(\R^n\times\R,\cN),\\
\f{1}{2}|\na v_j(x,t)|^2\ud x\ud t\wc^*\mu:=\f{1}{2}|\na v(x,t)|^2\ud x\ud t+\nu\text{ in }\M(\R^n\times\R).
\end{gathered}\label{vjconvergence}
\ee
Moreover, we assume that as $ j\to+\ift $, $
E_j\to E\in[0,E'] $, $ \f{y_{i,j}}{r_j}\to y_i\in\ol{B}_1 $, $ \wt{L}_j\to L $, and $ \f{x_j}{r_j}\to x\in\ol{B}_1\cap L $. The assumptions \eqref{Phiuj1} and \eqref{Phiujyj} imply that $ \cW(u_j,(y_{i,j},0),2r_j,2\eta_{j}r_j)\to 0^+ $. Due to Proposition \ref{BlowUplimits}\ref{BlowUplimits6}, $ \mu $ is backward self-similar at each $ y_i $ in $ \R^n\times(-16,0) $. By  Lemma \ref{ConeSplittingCor}, $ \mu $ is invariant under translations along $ V=\op{span}\{y_j-y_0\}_{j=1}^k $. Taking $ j\to+\ift $ in \eqref{PhiujEjeta}, we deduce that
\[
\Phi(\mu,(x,0),2\ga\rho)=\f{1}{2}\int_{T_{2\ga\rho}((x,0))}G_{(x,0)}(y,s)\ud\mu(y,s)\leq E-\eta'.
\]
Since $ x\in L $, from the $ V $-invariance of $ \mu $ it follows that $ \Phi(\mu,(x,0),2\ga\rho)=E $, which leads to a contradiction to $ \eta'>0 $. Thus, the first property holds.

\emph{Step 2. Proof of the property \eqref{SubsetSg}.} Fix $ \beta $ from Step 1 and choose $ \eta_0 $ small enough. Suppose that the inclusion fails, so there exists $ z_j\in \Sg_{\va;\eta_jr_j}^k(u_j,0)\cap B_{r_j}\backslash B_{\beta r_j}(L_j) $. Up to a subsequence, we still have \eqref{vjconvergence} and can assume that $ \eta_j\to 0^+ $, $ \f{y_{i,j}}{r_j} $, $ L_j $, and $ E_j $ converge as in Step 1. Moreover, we have $ \f{z_j}{r_j}\to z\in\ol{B}_1\backslash B_{\beta}(L) $. Since $ z\notin B_{\beta}(L) $, it follows from Proposition \ref{BlowUplimits}\ref{BlowUplimits6}, Remark \ref{Remarkmubound}, and analogous arguments in \cite[Section 3]{LW02b} that each blow-up limit of $ \mu $ at the point $ (z,0) $ (the limit of $ \{T_{(z,0),\lda}\mu\}_{\lda>0} $ up to a subsequence) exists and is also backward self-similar at $ 0^{n,1} $. Since $ \mu $ is already backward invariant with respect to $ V=\op{span}\{y_j-y_0\}_{j=1}^k $, Lemma \ref{ConeSplittingCor} implies that the blow-up limit of $ \mu $ at $ z $ will be spatially $ (k+1) $-symmetric at $ t=0 $ with respect to $ \op{span}\{z,V\} $. Let $ \mu^0 $ be a blow-up limit of $ \mu $ at $ (z,0) $. As a result, we can choose $ r>0 $ such that
\be
\MM_{0^{n,1},1}(T_{(z,0),r}\mu,\mu^0)<\f{\va}{4}.\label{dTz0rmu}
\ee
On the other hand, for $ j\in\Z_+ $ sufficiently large, we have
\[
\MM_{0^{n,1},1}\(\f{1}{2}|\na(T_{(r_j^{-1}z_j,0),r}T_{0^{n,1},r_j}u_j)|^2\ud y\ud s,T_{(z,0),r}\mu\)<\f{\va}{4}.
\]
This, together with \eqref{dTz0rmu}, implies that
\[
\MM_{0^{n,1},1}\(\f{1}{2}|\na(T_{(z_j,0),rr_j}u_j)|^2\ud y\ud s,\mu^0\)<\f{\va}{2}.
\]
It contradicts $ z_j\in\Sg_{\va;\eta_jr_j}^k(u_j,0) $ when $ \eta_j<r $, proving the inclusion.
\end{proof}
Now we use Lemma \ref{DichotomyLem} to give the proof of Proposition \ref{Dichotomy}. 
\begin{proof}[Proof of Proposition \ref{Dichotomy}]
Assuming that there exist points $ \{y_i\}_{i=0}^k $ satisfying the conditions of the Lemma \ref{DichotomyLem}, we can choose $ \eta,\beta>0 $ small enough to ensure that the first property holds. If no such points exist, the second property follows directly, as we can contain the set within a lower-dimensional affine subspace.
\end{proof}

\section{The proof of Lemma \ref{MainCovering}}\label{ProofofCoverSection}

Without loss of generality, we assume $R \in (0, \frac{1}{10})$. If $R\in[\f{1}{10},1)$, we can trivially cover $\Sg_{\va;\eta Rr}^k(u,t_0) \cap B_r(x_0)$ with balls $\{B_{Rr}(y)\}_{y \in \cC}$ such that $\cC \subset \Sg_{\va;\eta Rr}^k(u,t_0) \cap B_r(x_0)$ and $\#\cC \leq C(n)$, which implies $(\#\cC)(Rr)^k \leq C(n)r^k$ and the result holds.

Let $ \eta,\eta',\ga,r $, and $ \rho $ be positive numbers in $ (0,\frac{1}{10}) $, to be specified later. We require $ \eta\ll\rho $ in the sense of Remark \ref{Conventionll}. Define $ r_i=\rho^ir $. In this section, all arguments assume the conditions of Lemma \ref{MainCovering}, particularly $ r\in(0,\eta^2) $. We start by defining good and bad balls.

\begin{defn}\label{DeftypeItypeII}
For $ s\in[Rr,r) $ and $ y\in B_r(x_0) $, we call $ B_s(y) $ a good ball if
\[
\Phi(u,(z,t_0),2\ga\rho s)\geq E-\eta'\quad\text{for any }z\in\Sg_{\va;\eta Rr}^k(u,t_0)\cap B_s(y),
\]
where $ E $ is given by \eqref{Edefinition}. Otherwise, $ B_s(y) $ is a bad ball.
\end{defn}

\begin{lem}\label{badballproperty}
Suppose that $ B_s(y) $ is a bad ball. There exists $ \eta\in(0,\f{1}{10}) $, depending only on $ \va,\eta',\ga,\Lda,n,\cN $, and $ \rho $ such that if $ r\in(0,\eta^2) $, then
\[
\left\{z\in B_s(y):\Phi(u,(z,t_0),2\eta s)\geq E-\f{\eta}{2}\right\}\subset B_{\rho s}(L)
\]
for some $ L=L(y,s)\in\bA(n,k-1) $.
\end{lem}

\begin{proof}
Consider a bad ball $ B_s(y) $. By Proposition \ref{Monotonicity}, if $ \eta\in(0,\f{1}{10}) $ is sufficiently small and $ r\in(0,\eta^2) $, then we have
\begin{align*}
\sup_{z\in B_s(y)}\Phi(u,(z,t_0),2s)&\leq\sup_{z\in B_{2r}(x_0)}\Phi(u,(z,t_0),2r)+C(\Lda,n)r\leq E+\f{\eta}{2}.
\end{align*}
The result is a direct consequence of Proposition \ref{Dichotomy}.
\end{proof}

\subsection{Building trees at good balls}

Let $ A\in\Z_{\geq 0} $, and assume that $ B_{r_A}(a) $ is a good ball. We aim to construct a good tree $ \cT_G:=\cT_G(B_{r_A}(a)) $, defined as a collection of balls $ \cT_G(B_{r_A}(a)):=\{B_{r_y}(y)\}_{y\in\cC_G(B_{r_A}(a))} $ that covers $ \Sg_{\va;\eta Rr}^k(u,t_0)\cap B_r(x_0)\cap B_{r_A}(a) $. The centers satisfy 
\[
\cC_G(B_{r_A}(a))\subset\Sg_{\va;\eta Rr}^k(u,t_0)\cap B_{r_A}(a).
\]
To build $ \cT_G(B_{r_A}(a)) $, we inductively define, for each scale $ r_i $ with $ i\in\Z_{\geq A} $, a collection of balls $ \{B_{r_i}(y)\}_{y\in\cB_i\cup\cG_i\cup\cS_i} $, where $ \cB_i $, $ \cG_i $, and $ \cS_i $ are disjoint sets of ball centers classified by specific criteria.

At the initial scale $ r_A $, set $ \cG_A=\{a\} $ and $ \cB_A=\cS_A=\emptyset $. Assuming that the construction is defined for scale $ r_{i-1} $, we proceed to scale $ r_i $. Let $ J_i $ be a maximal $ \f{2r_i}{5} $-net in
\[
\Sg_{\va;\eta Rr}^k(u,t_0)\cap B_r(x_0)\cap B_{r_A}(a)\cap B_{r_{i-1}}(\cG_{i-1})\backslash\bigcup_{j=A}^{i-1}B_{r_j}(\cB_j).
\]
\begin{itemize}
\item If $ r_i\leq Rr $, set $ \cS_i=J_i $, $ \cG_i=\cB_i=\emptyset $, and stop the process.
\item If $ r_i>Rr $, classify $ J_i $ as $ \cG_i\cup\cB_i$, where $ z\in\cG_i $ if $ B_{r_i}(z) $ is a good ball, and $ z\in\cB_i $ otherwise.
\end{itemize}
Since $ r_i=\rho^ir $ decreases geometrically, the process ends after finitely many steps. The construction decomposes good balls at each step when $ r_i>Rr $, retaining bad balls. In $ \cT_G $, centers not in some $ \cS_i $ correspond to bad balls. It motivates the following definition.

\begin{defn}\label{goodtreeclass}
For the good tree $ \cT_G=\cT_G(B_{r_A}(a)) $, we define
\[
\cF(\cT_G):=\bigcup_{i}\cB_i\quad\text{and}\quad\cS(\cT_G):=\bigcup_{i}\cS_i.
\]
Balls $ \{B_{r_y}(y)\}_{y\in\cF(\cT_G)} $ are called final balls and $ \{B_{r_y}(y)\}_{y\in\cS(\cT_G)} $ are stop balls. Note that $ \cS_i\neq\emptyset $ only if $ r_i\leq Rr $ and $ r_{i-1}>Rr $.
\end{defn}

\begin{lem}\label{lemgoodtree}
Let $ \cT_G=\cT_G(B_{r_A}(a)) $ be a good tree. The following properties hold.
\begin{enumerate}[label=$(\theenumi)$]
\item\label{good1} There is $ C>0 $, depending only on $ n $ such that 
\be
\sum_{y\in\cF(\cT_G)\cup\cS(\cT_G)}r_y^k\leq Cr_A^k.\label{estimategoodtree}
\ee
\item\label{good2} For any $ y\in\cS(\cT_G) $, $ \rho Rr<r_y\leq Rr $.
\item\label{good3} We have the covering
\[
\Sg_{\va;\eta Rr}^k(u,t_0)\cap B_r(x_0)\cap B_{r_A}(a)\subset\bigcup_{y\in\cF(\cT_G)\cup\cS(\cT_G)}B_{r_y}(y).
\]
\end{enumerate}
\end{lem}
\begin{proof}
The maximality of $ J_i $ at each step ensures that $ \{B_{\f{r_i}{5}}(y)\}_{y\in\cF(\cT_G)\cup\cS(\cT_G)} $ are pairwise disjoint. Since $ \cF(\cT_G)\cup\cS(\cT_G)\subset\Sg_{\va;\eta Rr}^k(u,t_0) $, for any $ y\in\cS_i $ or $ y\in\cB_i $, there exists $ y'\in\cG_{i-1} $ such that $ y\in B_{r_{i-1}}(y') $. By Definition \ref{DeftypeItypeII}, for $ y\in\cB_i $,
\be
\Phi(u,(y,t_0),\ga r_i)=\Phi(u,(y,t_0),\ga\rho r_{i-1})\geq E-\eta',\label{cBiass}
\ee
and for any $ y\in\cS_i $,
\be
\Phi(u,(y,t_0),\ga r_i)\geq E-\eta'.\label{cSiass}
\ee
Using Proposition \ref{Monotonicity} with sufficiently small $ \eta=\eta(\va,\eta',\ga,\Lda,n,\rho)>0 $ and $ r\in(0,\eta^2) $, we have
\[
\sup_{y\in B_{r_A}(a)}\Phi(u,(y,t_0),2r_A)\leq E+\eta'.
\]
Combining \eqref{cBiass} and \eqref{cSiass} and choosing suitable $ \eta',\eta >0 $, Lemma \ref{ReifenbergLemma} yields \eqref{estimategoodtree}. The second property follows from Definition \ref{goodtreeclass}, as $ \cS_i=\emptyset $ unless $ r_i\leq Rr $ and $ r_{i-1}>Rr $. For the last property, we show inductively that for $ i\in\Z_{\geq A} $,
\be
\Sg_{\va;\eta Rr}^k(u,t_0)\cap B_r(x_0)\cap B_{r_A}(a)\subset B_{r_i}(\cG_i)\cup\bigcup_{j=A}^iB_{r_j}(\cB_j\cup\cS_j).\label{inductiveSg}
\ee
It holds trivially at $ i=A $. Assuming that it holds for $ i-1 $, the construction of $ J_i $ ensures that
\[
\Sg_{\va;\eta Rr}^k(u,t_0)\cap B_r(x_0)\cap B_{r_{i-1}}(\cG_{i-1})\backslash\bigcup_{j=A}^{i-1}B_{r_j}(\cB_j)\subset B_{r_i}(\cB_i\cup\cG_i\cup\cS_i).
\]
proving the step for $ i $ and thus the covering property \eqref{inductiveSg}.
\end{proof}

\subsection{Building trees in bad balls}

Suppose $ B_{r_A}(a) $ is a bad ball at scale $ r_A $, where $ A\in\Z_{\geq 0} $. We aim to construct a bad tree $ \cT_B:=\cT_B(B_{r_A}(a)) $ for it. Similar to the good tree, we define $ \cT_B(B_{r_A}(a)):=\{B_{r_y}(y)\}_{y\in\cC(\cT_B(B_{r_A}(a)))} $, where the centers satisfy $ \cC(\cT_B(B_{r_A}(a)))\subset\Sg_{\va;\eta Rr}^k(u,t_0)\cap B_{r_A}(a) $, and the balls cover $ \Sg_{\va;\eta Rr}^k(u,t_0)\cap B_r(x_0)\cap B_{r_A}(a) $. Following the approach used for the constructions of good trees in the previous subsection, we inductively build a collection of balls for each scale $ r_i $ with $ i\in\Z_{\geq A} $ and stop at the desired step. 

At the initial scale $ i=A $, set $ \cB_A=\{a\} $ and $ \cG_A=\cS_A=\emptyset $. Now, assume $ \cB_j,\cG_j $, and $ \cS_j $ have been given for scales $ r_j $ with $ j\in\Z\cap[A,i-1] $. We construct the collection for scale $ r_i $ as follows.
\begin{itemize}
\item If $ r_i\leq Rr $, set $ \cG_i=\cB_i=\emptyset $ and let $ \cS_i $ be a maximal $ \f{2\eta r_{i-1}}{5} $-net in
\[
\Sg_{\va;\eta Rr}^k(u,t_0)\cap B_r(x_0)\cap B_{r_A}(a)\cap B_{r_{i-1}}(\cB_{i-1}).
\]
Per Remark \ref{Conventionll}, we choose $ \eta\ll\rho $, ensuring $ \eta r_{i-1}<r_i\leq Rr $.
\item If $ r_i>Rr $, for any $ y\in\cB_{i-1} $, Lemma \ref{badballproperty} provides $ L(y,r_{i-1})\in\bA(n,k-1) $. Define $ \cS_i $ as a maximal $ \f{2\eta r_{i-1}}{5} $-net in
\be
\Sg_{\va;\eta Rr}^k(u,t_0)\cap B_r(x_0)\cap B_{r_A}(a)\cap\bigcup_{y\in\cB_{i-1}}(B_{r_{i-1}}(y)\backslash B_{2\rho r_{i-1}}(L(y,r_{i-1}))),\label{Jibadchoose}
\ee
and define $ J_i $ as a maximal $ \f{2r_i}{5} $-net in
\[
\Sg_{\va;\eta Rr}^k(u,t_0)\cap B_r(x_0)\cap B_{r_A}(a)\cap\bigcup_{y\in\cB_{i-1}}(B_{r_{i-1}}(y)\cup B_{2\rho r_{i-1}}(L(y,r_{i-1}))).
\]
Split $ J_i $ into $\cG_i$ and $\cB_i$, where $ y\in\cG_i $ if $ B_{r_i}(y) $ is a good ball and $ y\in\cB_i $ otherwise.
\end{itemize} 

Again, since $ r_i=\rho^ir $, the inductive arguments above will end after finite steps and give the construction of a bad tree on a bad ball $ B_{r_A}(a) $. Similar to Definition \ref{goodtreeclass}, in $ \cT_B $, if a center $ y\in\cC(\cT_B) $ does not belong to any $ \cS_i $, then $ B_{r_y}(y) $ is a good ball, so it is natural to give the definition below.

\begin{defn}
For a bad tree $ \cT_B=\cT_B(B_{r_A}(a)) $, we let
\[
\cF(\cT_B):=\bigcup_{i}\cG_i\quad\text{and}\quad\cS(\cT_B):=\bigcup_{i}\cS_i.
\]
We call balls in $ \{B_{r_y}(y)\}_{y\in\cF(\cT_B)} $ final balls in $ \cT_B $, and call balls in $ \{B_{r_y}(y)\}_{y\in\cS(\cT_B)} $ stop balls in $ \cT_B $.
\end{defn}

\begin{lem}\label{badtree}
Let $ \cT_B=\cT_B(B_{r_A}(a)) $ be a bad tree. The following properties hold. 
\begin{enumerate}[label=$(\theenumi)$]
\item\label{bad1} There exist $ C>0 $ depending only on $ n $ and $ C'>0 $ depending only on $ \eta,n $ such that
\be
\sum_{y\in\cF(\cT_B)}r_y^k\leq C\rho r_A^k\quad\text{and}\quad\sum_{y\in\cS(\cT_B)}r_y^k\leq C'r_A^k.\label{badtreeestimate}
\ee
\item\label{bad2} We have the covering property
\[
\Sg_{\va;\eta Rr}^k(u,t_0)\cap B_r(x_0)\cap B_{r_A}(a)\subset\bigcup_{y\in\cF(\cT_B)\cup\cS(\cT_B)}B_{r_y}(y).
\]
\item\label{bad3} For any $ y\in\cS(\cT_B) $, either $ \eta Rr\leq r_y\leq Rr $ or
\[
\sup_{z\in B_{2r_y}(y)}\Phi(u,(z,t_0),2r_y)\leq E-\f{\eta}{3}.
\]
\end{enumerate}
\end{lem}
\begin{proof}
For $ r_i>Rr $, the centers in $ \cG_i\cup\cB_i $ lie within $ B_{2\rho r_{i-1}}(L(y,r_{i-1})) $, where $ L\in\bA(n,k-1) $ is from Lemma \ref{badballproperty}. The balls $ \{B_{\f{r_i}{5}}(y)\}_{y\in\cG_i\cup\cB_i} $ are pairwise disjoint due to the maximality of $ J_i $. For a bad ball $ B_{r_{i-1}}(y) $ with $ y\in\cB_{i-1} $, we estimate
\[
\#[(\cG_i\cup\cB_i)\cap B_{r_{i-1}}(y)]\leq\f{\w_{k-1}\w_{n-k+1}(3\rho)^{n-k+1}}{\w_n\(\f{\rho}{5}\)^n}\leq C_0(n)\rho^{1-k}.\label{GiBileq}
\]
Iterating this in the construction process, we get
\[
\#(\cG_i\cup\cB_i)r_i^k\leq C_0\rho(\#\cB_{i-1})r_{i-1}^k\leq C_0\rho[\#(\cB_{i-1}\cup\cG_{i-1})]r_{i-1}^k\leq...\leq (C_0(n)\rho)^{i-A}r_A^k.
\]
Choosing $ \rho=\rho(n)>0 $ small enough, we sum over $ i\in\Z_{\geq A+1} $ and obtain
\[
\sum_{i\in\Z_{\geq A+1}}\#(\cG_i\cup \cB_i)r_i^k\leq\sum_{i\in\Z_{\geq A+1}}(C_0\rho)^{i-A}r_A^k\leq C(n)\rho r_A^k.
\]
Since $ \cF(\cT_B)=\cup_i\cG_i $, it gives the first estimate of \eqref{badtreeestimate}. For $ \cS_i $ with $ i\in\Z_{\geq A+1} $, the balls $ \{B_{\eta r_{i-1}}(y)\}_{y\in\cS_i} $ form a Vitali collection in $ B_{r_{i-1}}(\cB_{i-1}) $, so $
\#\cS_i\leq C(n)\eta^{-n}(\#\cB_{i-1}) $. For $ i=A $, $ \cS_A=\emptyset $, so
\[
\sum_{i\in\Z_{\geq A+1}}(\#\cS_i)(\eta r_{i-1})^k\leq 10^n\eta^{k-n}\(\sum_{i\in\Z_{\geq A}}(\#\cB_i)r_i^k\)\leq C(n,\eta)r_A^k,
\]
yielding the second estimate of \eqref{badtreeestimate}. 

The covering property follows as in Lemma \ref{lemgoodtree}\ref{good3}, so we omit the details for brevity. For the third property, consider $y \in\cS_i $ with $ r_y=\eta r_{i-1} $. If $ r_i>Rr $, then 
\[
y\in B_{r_{i-1}}(y')\backslash B_{2\rho r_{i-1}}(L(y',r_{i-1})) 
\]
for some $ y'\in\cB_{i-1} $. By Lemma \ref{badballproperty}, with $ \eta>0 $ small ($ \eta<\f{1}{2}\rho $) and $ r\in(0,\eta^2) $,
\[
\sup_{z\in B_{2r_y}(y)}\Phi(u,(z,t_0),2r_y)\leq\sup_{z\in B_{\rho r_{i-1}}(y)}\Phi(u,(z,t_0),2\eta r_{i-1})+r\leq E-\f{\eta}{2}+\eta^2\leq E-\f{\eta}{3}.
\]
If $ r_i\leq Rr $, then $ r_{i-1}\geq Rr $, and since $ r_y=\eta r_{i-1} $ and $ \eta\ll\rho $,
\[
Rr\geq\rho r_{i-1}\geq\eta r_{i-1}=r_y\geq \eta Rr,
\]
completing the proof.
\end{proof}

\subsection{Proof of Lemma \ref{MainCovering}}

Under the assumptions of Lemma \ref{MainCovering}, we inductively construct, for any $ i\in\Z_{\geq 0} $, a collection of balls $ \{B_{r_y}(y)\}_{y\in\cF_i}\cup\{B_{r_y}(y)\}_{y\in\cS_i} $ that covers $ \Sg_{\va;\eta Rr}^k(u,t_0)\cap B_r(x_0) $. It builds on the good and bad tree constructions from previous subsections. At each step $ i $, the balls in $ \{B_{r_y}(y)\}_{y\in\cF_i} $ are either all good or all bad. We define a stopping condition, and the final collection will satisfy the requirements of the lemma.

For $ i=0 $, set $ \cF_0=\{x_0\} $ with $ r_{x_0}=r $ and $ \cS_0=\emptyset $. Thus, $ B_r(x_0) $ is the only initial ball, either good or bad. Assuming that the collection is defined up to step $ i-1 $, where all balls are centered in $ \cF_{i-1} $, are either all good or all bad, we proceed as follows.
\begin{itemize}
\item If all balls in $ \cF_{i-1} $ are good, define a good tree $ \cT_{G,y}=\cT_G(B_{r_y}(y)) $ for any $ y\in\cF_{i-1} $. Set
\[
\cF_i=\bigcup_{y\in\cF_{i-1}}\cF(\cT_{G,y})\quad\text{and}\quad\cS_i=\cS_{i-1}\cup\bigcup_{y\in \cF_{i-1}}\cS(\cT_{G,y}).
\]
Since the final balls of good trees are always bad balls, we see that all the balls with the center in $ \cF_i $ are bad. 
\item If all balls in $ \cF_{i-1} $ are bad, define a bad tree $ \cT_{B,y}=\cT_B(B_{r_y}(y)) $ for any $ y\in\cF_{i-1} $. Set
\[
\cF_i=\bigcup_{y\in\cF_{i-1}} \cF(\cT_{B,y})\quad\text{and}\quad\cS_i=\cS_{i-1}\cup\bigcup_{y\in\cF_{i-1}}\cS(\cT_{B,y}).
\]
Here, the final balls of the bad trees are good, so all the balls in $ \cF_i $ are good.
\end{itemize}

We have now completed the construction. The following lemma summarizes its properties.

\begin{lem}\label{lemmafinal}
There exists $ N\in\Z_+ $ such that $ \cF_N=\emptyset $, and the following properties hold.
\begin{enumerate}[label=$(\theenumi)$]
\item There exists $ C>0 $, depending only on $ n $ such that
\be
\sum_{i=0}^{N-1}\sum_{y\in\cF_i}r_y^k\leq Cr^k.\label{Finalballestimate}
\ee
\item We have the estimate
\be
\sum_{y\in\cS_N}r_y^k\leq C'r^k,\label{Stopballestimate}
\ee
where $ C'>0 $ depends only on $ \va,\Lda,n $, and $ \cN $.
\item The balls in $ \{B_{r_y}(y)\}_{y\in\cS_N} $ form a covering of $ \Sg_{\va;\eta Rr}^k(u,t_0)\cap B_r(x_0) $, that is,
\[
\Sg_{\va;\eta Rr}^k(u,t_0)\cap B_r(x_0)\subset\bigcup_{y\in\cS_N}B_{r_y}(y).
\]
\item For any $ y\in\cS_N $, either $ \eta Rr\leq r_y\leq Rr $ or
\[
\sup_{z\in B_{2r_y}(y)}\Phi(u,(z,t_0),2r_y)\leq E-\f{\eta}{3}.
\]
\end{enumerate}
\end{lem}

The lemma above directly implies Lemma \ref{MainCovering}.

\begin{proof}[Proof of Lemma \ref{MainCovering} given Lemma \ref{lemmafinal}]
Take $ \cC=\cS_N $. For any $ y\in\cC $, if $ r_y>Rr $, keep $ r_y $ unchanged; if $ r_y\leq Rr $, set $ r_y=Rr $. The desired properties follow from Lemma \ref{lemmafinal}.
\end{proof}

\begin{proof}[Proof of Lemma \ref{lemmafinal}]
First, we show that $ \cF_N=\emptyset $ for some $ N\in\Z_+ $. In both good and bad tree constructions for a ball $ B_{r_A}(a) $, the radius of a final ball satisfies $ r_y\leq\rho r_A $. Thus,
\[
\max_{y\in\cF_i}r_y\leq\rho\max_{y\in\cF_{i-1}}r_y\leq C\rho^ir.
\]
For large $ i\in\Z_+ $, $ \rho^ir<Rr $, contradicting the construction requirement that final balls have radii $ \geq Rr $. Hence, the process terminates when $ \cF_N=\emptyset $.

For the first property, consider the following two cases.
\begin{itemize}
\item If balls in $ \cF_i $ are good, they are final balls of bad trees from $ \cF_{i-1} $. By Lemma \ref{badtree}\ref{bad1}, we have
\be
\sum_{y\in\cF_i}r_y^k\leq C(n)\rho\(\sum_{y\in \cF_{i-1}}r_y^k\).\label{Use1}
\ee
\item If balls in $\cF_i$ are bad, they are final balls of good trees from $ \cF_{i-1} $. Using Lemma \ref{lemgoodtree}\ref{good1}, we obtain that
\be
\sum_{y\in \cF_i} r_y^k\leq C(n)\(\sum_{y\in \cF_{i-1}}r_y^k\).\label{Use2}
\ee
\end{itemize}
Since good trees produce bad final balls and bad trees produce good ones, the estimates \eqref{Use1} and \eqref{Use2} alternate. Combining these, we get
\[
\sum_{y\in \cF_i}r_y^k\leq C(n)(C_0(n)\rho)^{\f{i}{2}}\leq C(n)2^{-\f{i}{2}}r^k,
\]
assuming $ \rho>0 $ is chosen small enough. Summing from $ i=0 $ to $ N-1 $ yields \eqref{Finalballestimate}.

To show the estimate \eqref{Stopballestimate}, we note that for any $ y\in\cS_N $, the ball $ B_{r_y}(y) $ is produced from a good or bad tree for a ball with center $ y'\in \cF_i $, for some $ i<N $. Using Lemma \ref{lemgoodtree}\ref{good1} and Lemma \ref{badtree}\ref{bad1}, together with the first point of this lemma, we have
\[
\sum_{y\in\cS_N}r_y^k\leq C(\va,n,\Lda,\cN)\(\sum_{i=0}^N\sum_{y\in \cF_i} r_y^k\)\leq C(\va,n,\Lda,\cN)r^k.
\]

Regarding the third property of the current lemma, we apply Lemma \ref{lemgoodtree}\ref{good3} and Lemma \ref{badtree}\ref{bad2} to each tree constructed at balls centered in $ \cF_{i-1} $. It gives that
\[
\bigcup_{y\in\cF_{i-1}}(\Sg_{\va;\eta Rr}^k(u,t_0)\cap B_{r_y}(y))\subset \bigcup_{y\in\cF_i}(\Sg_{\va;\eta Rr}^k(u,t_0)\cap B_{r_y}(y))\cup\bigcup_{y \in \cS_i} B_{r_y}(y)
\]
Since $ \Sg_{\va;\eta Rr}^k(u,t_0)\cap B_r(x_0)\subset B_r(x_0) $, induction shows that for any $ i\in\Z\cap[1,N] $,
\[
\Sg_{\va;\eta Rr}^k(u,t_0)\cap B_r(x_0)\subset\bigcup_{y\in\cF_i}B_{r_y}(y) \cup\bigcup_{y\in \cS_i}B_{r_y}(y).
\]
At $ i=N $, since $ \cF_N=\emptyset $, the covering holds. The fourth property follows directly from Lemma \ref{lemgoodtree}\ref{good2} and Lemma \ref{badtree}\ref{bad3}.
\end{proof}

\section{Proof of main theorems}\label{ProofofMainTheorems}
In this section, we give the proof of the main theorems. First, we provide a volume estimate for quantitative spatial stratification.
\begin{thm}[Main theorem of quantitative spatial stratification]\label{MainStratification}
Let $ \va\in(0,1) $, $ \Lda>0 $, $ k\in\Z\cap[0,n-1] $, $ R\in(0,1] $, and $ t\in(-4,4) $. Assume that $ u\in H_{\Lda}(P_4,\cN) $. There exists $ C>0 $ that depends only on $ \va,\Lda,n,\cN $, and $ R $ such that the following properties hold. 

For any $ r\in(0,R) $,
\begin{align}
\cL^n(B_r(\Sg_{\va;r,R}^k(u,t)\cap B_1))&\leq Cr^{n-k},\label{MinkowskiContentSigma}\\
\cL^n(B_r(\Sg_{\va;0,R}^k(u,t)\cap B_1))&\leq Cr^{n-k}.\label{MinkowskiContentSigma2}
\end{align}
Moreover, for any $ x\in B_1 $ and $ \rho\in(0,1) $,
\be
\HH^k(\Sg_{\va;0,R}^k(u,t)\cap B_{\rho}(x))\leq C\rho^k.\label{AlforsEstimates}
\ee
\end{thm}
Before giving the proof of Theorem \ref{MainStratification}, we first show that we can use Lemma \ref{InculsionDifferentScales} to reduce the proof to a special case.
\begin{lem}\label{reduceR1}
Theorem \ref{MainStratification} holds for any $ R\in(0,1] $ if it holds for $ R=1 $.
\end{lem}
\begin{proof}
Assuming that Theorem \ref{MainStratification} is satisfied for $ R=1 $, we now show that the theorem holds for any $ R\in(0,1) $. For \eqref{MinkowskiContentSigma}, without loss of generality, we assume that $ r<\f{R}{2} $ since otherwise the inequality is ensured trivially with the help of
\[
\cL^n(B_r(\Sg_{\va;r,R}^k(u,t)\cap B_1))\leq \cL^n(B_1)\leq C(n,R)r^{n-k}.
\]

If $ R\in(\f{1}{2},1) $, then \eqref{deltarRkut} implies the existence of $ \delta=\delta(\va,\Lda,n,\cN)>0 $ such that
\be
\Sg_{\va;r,R}^k(u,t)\subset\Sg_{\va;r,\f{1}{2}}^k(u,t)\subset\Sg_{\delta;r,1}^k(u,t),\label{ChainArguments1}
\ee
where the second inclusion follows from Remark \ref{inclusionSvak}. Applying Theorem \ref{MainStratification} for $ R=1 $ with $ \va $ replaced by $ \delta $, the result follows. 

If $ R\in(0,\f{1}{2}] $, there exists $ \ell=\ell(R)\in\Z_+ $ such that $ \f{1}{2}\leq 2^{\ell}R<1 $. By using Lemma \ref{InculsionDifferentScales}  $ \ell $ times, we have
\be
\Sg_{\va;r,R}^k(u,t)\subset\Sg_{\delta_1;r,2R}^k(u,t)\subset\Sg_{\delta_2;r,2^2R}^k(u,t)\subset...\subset\Sg_{\delta_{\ell};r,2^{\ell}R}^k(u,t).\label{ChainArguments2}
\ee
Note that here $ \delta_{\ell}=\delta_{\ell}(\va,\Lda,n,\cN,R)>0 $. It reduces to the case where $ R>\f{1}{2} $. The property \eqref{MinkowskiContentSigma2} follows directly from \eqref{MinkowskiContentSigma}. When it comes to \eqref{AlforsEstimates}, the general case with $ R\in(0,1) $ follows directly from the chain properties given in \eqref{ChainArguments1} and \eqref{ChainArguments2} with similar arguments.
\end{proof}
Next, we use Proposition \ref{MainCoveringCor} to give the proof of Theorem \ref{MainStratification}.
\begin{proof}[Proof of Theorem \ref{MainStratification}:] By Lemma \ref{reduceR1}, it suffices to prove the case $R = 1$. Fix $ s\in (0,\f{1}{10}) $. Choose $ \eta=\eta(\va,\Lda,n,\cN)>0 $ as given by  Proposition \ref{MainCoveringCor}. We cover $ \Sg_{\va;\frac{\eta^3s}{2},1}^k(u,t) $ with balls $ \{B_{\f{\eta^2 s}{2}}(x_i)\}_{i=1}^{N_0} $, where $ \{x_i\}_{i=1}^{N_0}\subset B_1 $ and $ 1\leq N_0\leq C(\va,\Lda,n,\cN)s^{-k} $. Combining Proposition \ref{MainCoveringCor}  with the bound on $ N_0 $, we get
\be
\cL^n\left[B_{\f{\eta^2s}{2}}\(\Sg_{\va;\f{\eta^3s}{2},1}^k(u,t)\cap B_1\)\right]\leq C(\va,\Lda,n,\cN)\(\f{\eta^2s}{2}\)^{n-k}.\label{Brprimes}
\ee
For $ 0<r<\f{\eta^3}{20} $, choose $ s\in(0,\f{1}{10}) $ such that $ r=\f{\eta^3s}{2} $. It implies that
\be
\cL^n(B_r(\Sg_{\va;r,1}^k(u,t)\cap B_1))\leq C(\va,\Lda,n,\cN)r^{n-k}.\label{Lnleq}
\ee
If $ \f{\eta^3}{20}\leq r<1 $, then
\[
\cL^n(B_{r}(\Sg_{\va;r,1}^k(u,t)\cap B_1))\leq\cL^n(B_1)\leq C(\va,\Lda,n,\cN)r^{n-k}.
\]
This, together with \eqref{Lnleq}, implies \eqref{MinkowskiContentSigma}, which in turn suggests \eqref{MinkowskiContentSigma2}.

Now, let $ \rho\in(0,\eta^2) $ and $ x\in B_1 $. For $ s\in(0,\f{\rho}{10}) $,  Proposition \ref{MainCoveringCor} implies that the $ k $-dimensional Hausdorff measure at the scale $ s $ satisfies
\be
\HH_s^k(\Sg_{\va;0,1}^k(u,t)\cap B_{\rho}(x))\leq C(n)(\#\cC)s^k\leq C(\va,\Lda,n,\cN)\rho^k.\label{Ahlfors1}
\ee
For $ \rho\geq\eta^2 $, the result is trivial by covering the ball $ B_{\rho}(x) $. Taking $ s\to 0^+ $, we obtain \eqref{AlforsEstimates}.
\end{proof}
\subsection{Proof of Theorem \ref{MainTheorem}}

Without loss of generality, let $ r=2 $ and $ X_0=0^{n,1} $. \eqref{Topchain} follows from Proposition \ref{SgnSgn2}. Regarding \eqref{MinkowskiContentSgn2}, it is a direct consequence of Theorem \ref{MainStratification} and Lemma \ref{FinallLemmaproperty}. Fix $ t\in(-r^2,+r^2) $ and $ k\in\Z\cap[0,n] $. We show that $ \Sg^k(u,t) $ is $ k $-rectifiable. By Remark \ref{inclusionSvak}, Remark \ref{decomSkuseSva} and covering arguments, it suffices to prove the rectifiability of $ \Sg_{\va;0,1}^k(u,t)\cap B_{\f{1}{2}} $ for any $ \va>0 $. Consider $ S\subset \Sg_{\va;0,1}^k(u,t)\cap B_{\f{1}{2}} $ with $ \HH^k(S)>0 $. For any $ x\in \Sg_{\va;0,1}^k(u,t) $ and $ 0<r\leq 1 $, define
\[
g(u,(x,t),r):=\Phi(u,(x,t),r)-\lim_{\rho\to 0^+}\Phi(u,(x,t),\rho).
\]
By the dominated convergence theorem, for any $ \sg>0 $, there is $ r_0=r_0(\sg,u)>0 $ such that 
\[
\f{1}{\HH^k(S)}\int_Sg(u,(x,t),r_0)\ud\HH^k(x)\leq\sg^2.
\]
Using an average argument, there is an $ \HH^k $-measurable set $ E\subset S $ such that $ \HH^k(E)\leq\sg\HH^k(S) $ and $ g(u,(x,t),r_0)\leq\sg $ for any $ x\in F:=S\backslash E $. Cover $ F $ with a finite collection of balls $ \{B_{r_0}(y_i)\}_{i=1}^{N_1} $ such that $ \{y_i\}_{i=1}^{N_1}\subset F $. 

We claim: if $ \sg=\sg(\va,\Lda,n,\cN)>0 $ is sufficiently small, then $ F\cap B_{r_0}(y_i) $ is $ k $-rectifiable for any $ i\in\Z\cap[1,N_1] $. 

If such a claim holds, then applying the process countably many times to $ S $ shows that $ S $ is $ k $-rectifiable. Since $ S $ was arbitrary, it follows that $ \Sg_{\va;0,1}^k(u,t)\cap B_{\f{1}{2}} $ is $ k $-rectifiable. 

To prove the claim, consider $ B_{r_0}(y_1) $ and assume that $ 0<r_0<\min\{\f{1}{100},\lda\} $, where $ \lda=\lda(\va,\Lda,n,\cN)>0 $ will be chosen later. For any $ z\in F $, $ g(u,(z,t),r_0)\leq\sg $. For any $ \sg'>0 $, by selecting sufficiently small $ (\sg,\lda)=(\sg,\lda)(\sg',\Lda,n,\cN)>0 $, Proposition \ref{Monotonicity} and Lemma \ref{pinchLem} imply that $ u $ is spatially $ (0,\sg') $-symmetric in $ B_{2s}(z) $ at $ t $ for any $ 0<s\leq\f{r_0}{2} $. For $ z\in F\subset \Sg_{\va;0,1}^k(u,t) $, $ u $ is not spatially $ (k+1,\va) $-symmetric in $ B_{s}(z) $ at $ t $. With $ \sg'=\sg'(\va,\Lda,n,\cN)>0 $ small enough, Proposition \ref{L2BestProp} yields that
\[
D_{\mu}^k(z,s)\leq \f{C(\va,\Lda,n,\cN)}{s^k}\int_{B_s(z)}\left[\cW\(u,(\zeta,t),s,\f{s}{2}\)+s\right]\ud\mu(\zeta)
\]
for any $ z\in F $ and $ 0<s\leq\f{r_0}{2} $, where $ \mu:=\HH^k\llcorner F $. Integrating both sides over $ B_r(x) $ with respect to $ z $ for $ x\in B_{r_0}(y_1) $ and $ 0<r\leq r_0 $, we obtain from Fubini theorem and \eqref{Ahlfors1} that
\[
\int_{B_r(x)}D_{\mu}^k(z,s)\ud\mu(z)\leq C(\va,\Lda,n,\cN)\int_{B_{r+s}(x)}\left[\cW\(u,(z,t),s,\f{s}{2}\)+s\right]\ud\mu(z).
\]
As a result, for any $ x\in B_{r_0}(y_1) $ and $ 0<r\leq r_0 $,
\begin{align*}
&\int_{B_r(x)}\(\int_0^rD_{\mu}^k(\cdot,s)\f{\ud s}{s}\)\ud\mu\leq C\left[\int_{B_{2r}(x)}\(\int_r^{2r}\Phi(u,(\cdot,t),s)\f{\ud s}{s}+1\)\ud\mu\right]\leq C(\va,\Lda,n,\cN)r^k,
\end{align*}
where the last step is based on \eqref{Ahlfors1}. Theorem \ref{Rei2} then ensures that $ F\cap B_{r_0}(y_1) $ is $ k $-rectifiable, proving the claim.

\subsection{Proof of Theorem \ref{ImprovementRegularity}} 

For $ t\in(-1,1) $, Lemma \ref{FinaLemma2} implies that
\[
\{x\in B_1:r_u((x,t))<\va r\}\subset\Sg_{\va;0,\va}^{n-3}(u,t)\cap B_1,
\]
where $ \va=\va(\Lda,n,\cN)>0 $. For $ r\in(0,\f{1}{10}) $, we get from Theorem \ref{MainStratification} that
\[
\cL^n(\{x\in B_1:|\na u(x,t)|>r^{-1}\})\leq\cL^n(\{x\in B_1:r_u((x,t))<r\})\leq C(\Lda,n,\cN)r^3.
\]
If $ r>\f{1}{10} $, then
\[
\cL^n(\{x\in B_1:|\na u(x,t)|>r^{-1}\})\leq\cL^n(B_1)\leq C(n)r^3.
\]
Consequently, 
\[ \|\na u(\cdot,t)\|_{L^{3,\ift}(B_1)}\leq C(\Lda,n,\cN),\]
completing the proof.

\section*{Acknowledgments}

The authors thank the referees for their careful reading of the manuscript and for their valuable comments and suggestions. The first two authors express their gratitude to Professor K. Wang for his generous and heartfelt support during their visit to Wuhan University. There, alongside the third author, they began exploring the problem addressed in this paper. The authors also thank Professor R. Haslhofer at the University of Toronto and Professor W. Jiang for their enlightening comments on the contexts of this paper. K. Wu is supported by the National Natural Science Foundation of China (No. 12401264). Z. Zhang is partially supported by the National Key R\&D Program of China under Grant 2023YFA1008801 and NSF of China under Grant 12288101.

\section*{Declarations} 

\subsection*{Data availability} This article has no associated data.
\subsection*{Conflict of interest} The authors declared that they have no conflict of interest.

\bibliographystyle{plain}

\begin{thebibliography}{10}

\bibitem{AT15}
J.~Azzam and X.~Tolsa.
\newblock Characterization of {{$n$}}-rectifiability in terms of {Jones}'
  square function: {Part} {II}.
\newblock {\em Geom. Funct. Anal.}, 25(5):1371--1412, 2015.

\bibitem{CDY92}
K.~Chang, W.~Ding, and R.~Ye.
\newblock Finite-time blow-up of the heat flow of harmonic maps from surfaces.
\newblock {\em J. Differential Geom.}, 36(2):507--515, 1992.

\bibitem{CHN13}
J.~Cheeger, R.~Haslhofer, and A.~Naber.
\newblock Quantitative stratification and the regularity of mean curvature
  flow.
\newblock {\em Geom. Funct. Anal.}, 23(3):828--847, 2013.

\bibitem{CHN15}
J.~Cheeger, R.~Haslhofer, and A.~Naber.
\newblock Quantitative stratification and the regularity of harmonic map flow.
\newblock {\em Calc. Var. Partial Differential Equations}, 53(1-2):365--381,
  2015.

\bibitem{CN13a}
J.~Cheeger and A.~Naber.
\newblock Lower bounds on {Ricci} curvature and quantitative behavior of
  singular sets.
\newblock {\em Invent. Math.}, 191(2):321--339, 2013.

\bibitem{CN13b}
J.~Cheeger and A.~Naber.
\newblock Quantitative stratification and the regularity of harmonic maps and
  minimal currents.
\newblock {\em Comm. Pure Appl. Math.}, 66(6):965--990, 2013.

\bibitem{CLL95}
Y.~Chen, J.~Li, and F.~Lin.
\newblock Partial regularity for weak heat flows into spheres.
\newblock {\em Comm. Pure Appl. Math.}, 48(4):429--448, 1995.

\bibitem{CW96}
Y.~Chen, J.~Li, and F.~Lin.
\newblock Partial regularity for weak heat flows into {Riemannian} homogeneous
  spaces.
\newblock {\em Comm. Partial Differential Equations}, 21(5-6):735--761, 1996.

\bibitem{CS89}
Y.~Chen and M.~Struwe.
\newblock Existence and partial regularity results for the heat flow for
  harmonic maps.
\newblock {\em Math. Z.}, 201(1):83--103, 1989.

\bibitem{Che91}
X.~Cheng.
\newblock Estimate of the singular set of the evolution problem for harmonic
  maps.
\newblock {\em J. Differential Geom.}, 34(1):169--174, 1991.

\bibitem{CM16}
T.~H. Colding and W. P. Minicozzi II.
\newblock The singular set of mean curvature flow with generic singularities.
\newblock {\em Invent. Math.}, 204(2):443--471, 2016.

\bibitem{Cor90}
J.-M. Coron.
\newblock Nonuniqueness for the heat flow of harmonic maps.
\newblock {\em Ann. Inst. H. Poincar\'{e} C Anal. Non Lin\'{e}aire},
  7(4):335--344, 1990.

\bibitem{DPW20}
J.~D\'{a}vila, M.~del Pino, and J.~Wei.
\newblock Singularity formation for the two-dimensional harmonic map flow into
  {{$S^{2}$}}.
\newblock {\em Invent. Math.}, 219(2):345--466, 2020.

\bibitem{EE19}
N.~Edelen and M.~Engelstein.
\newblock Quantitative stratification for some free-boundary problems.
\newblock {\em Trans. Amer. Math. Soc.}, 371(3):2043--2072, 2019.

\bibitem{ES64}
J. Eells Jr. and J.~H. Sampson.
\newblock Harmonic mappings of {Riemannian} manifolds.
\newblock {\em Amer. J. Math.}, 86:109--160, 1964.

\bibitem{Fel94}
M.~Feldman.
\newblock Partial regularity for harmonic maps of evolution into spheres.
\newblock {\em Comm. Partial Differential Equations}, 19(5-6):761–790, 1994.

\bibitem{Fre95a}
A.~Freire.
\newblock Uniqueness for the harmonic map flow in two dimensions.
\newblock {\em Calc. Var. Partial Differential Equations}, 3(1):95--105, 1995.

\bibitem{Fre95b}
A.~Freire.
\newblock Uniqueness for the harmonic map flow in two dimensions.
\newblock {\em Comment. Math. Helv.}, 70(2):310--338, 1995.

\bibitem{FWZ24}
H.~Fu, W.~Wang, and Z.~Zhang.
\newblock Quantitative stratification and sharp regularity estimates for
  supercritical semilinear elliptic equations.
\newblock Preprint, {arXiv}:2408.06726 [math.{AP}] (2024), 2024.

\bibitem{HSV19}
J.~Hirsch, S.~Stuvard, and D.~Valtorta.
\newblock Rectifiability of the singular set of multiple-valued energy
  minimizing harmonic maps.
\newblock {\em Trans. Am. Math. Soc.}, 371(6):4303--4352, 2019.

\bibitem{LW99}
F.~Lin and C.~Wang.
\newblock Harmonic and quasi-harmonic spheres.
\newblock {\em Comm. Anal. Geom.}, 7(2):397--429, 1999.

\bibitem{LW02b}
F.~Lin and C.~Wang.
\newblock Harmonic and quasi-harmonic spheres. {III}. {Rectifiability} of the
  parabolic defect measure and generalized varifold flows.
\newblock {\em Ann. Inst. H. Poincar\'{e} C Anal. Non Lin\'{e}aire},
  19(2):209--259, 2002.

\bibitem{LW10}
F.~Lin and C.~Wang.
\newblock On the uniqueness of heat flow of harmonic maps and hydrodynamic flow
  of nematic liquid crystals.
\newblock {\em Chinese Ann. Math. Ser. B}, 31(6):921--938, 2010.

\bibitem{Liu03}
X.~Liu.
\newblock Partial regularity for weak heat flows into a general compact
  riemannian manifold.
\newblock {\em Arch. Ration. Mech. Anal.}, 168(2):131--163, 2003.

\bibitem{NV17}
A.~Naber and D.~Valtorta.
\newblock Rectifiable-{Reifenberg} and the regularity of stationary and
  minimizing harmonic maps.
\newblock {\em Ann. Math. (2)}, 185(1):131--227, 2017.

\bibitem{NV18}
A.~Naber and D.~Valtorta.
\newblock Stratification for the singular set of approximate harmonic maps.
\newblock {\em Math. Z.}, 290(3-4):1415--1455, 2018.

\bibitem{Sim83}
L.~Simon.
\newblock {\em Lectures on geometric measure theory}, volume~3 of {\em Proc.
  Cent. Math. Anal. Aust. Natl. Univ.}
\newblock Australian National University, Centre for Mathematical Analysis,
  Canberra, 1983.

\bibitem{Str88}
M.~Struwe.
\newblock On the evolution of harmonic maps in higher dimensions.
\newblock {\em J. Differential Geom.}, 28(3):485--502, 1988.

\bibitem{Wan08}
C.~Wang.
\newblock Heat flow of harmonic maps whose gradients belong to {{$
  L_{x}^{n}L_{t}^{\ift}$}}.
\newblock {\em Arch. Ration. Mech. Anal.}, 18(2):351--369, 2008.

\bibitem{WWW25}
K.~Wang, J.~Wei, and K.~Wu.
\newblock Quantitative stratification for the fractional {Allen}-{Cahn}
  equation and stationary nonlocal minimal surface.
\newblock Preprint, {arXiv}:2503.16829 [math.{AP}] (2025), 2025.

\bibitem{WY24}
K.~Wang and G.~Yi.
\newblock Blow up analysis for a parabolic {MEMS} problem, {I}: {H}\"{o}lder
  estimate.
\newblock {\em Calc. Var. Partial Differential Equations}, 63(8):Paper No. 193,
  34 pp., 2024.

\bibitem{WZ24}
W.~Wang and Z.~Zhang.
\newblock Fine structure of rupture set for semilinear elliptic equation with
  singular nonlinearity.
\newblock Preprint, {arXiv}:2411.16048 [math.{AP}] (2024), 2024.

\bibitem{Whi97}
B.~White.
\newblock Stratification of minimal surfaces, mean curvature flows, and
  harmonic maps.
\newblock {\em J. Reine Angew. Math.}, 488:1--35, 1997.

\end{thebibliography}

\end{document}